\documentclass[11pt,oneside,reqno]{amsart}

\usepackage[latin1]{inputenc}
\usepackage[english]{babel}

\usepackage{amsfonts,latexsym}
\usepackage{amsmath}
\usepackage{amssymb}
\usepackage{amsthm}
\usepackage{mathtools}
\usepackage{comment}
\usepackage{soul}
\usepackage[dvips]{graphicx}

\newcommand{\N}{\mathbb{N}}

\newcommand{\R}{\mathbb{R}}
\newcommand{\Z}{\mathbb{Z}}

\newcommand{\C}{\mathbb{C}}

\newcommand\ton{\widehat}
\newcommand{\lf}{\left}
\newcommand{\rg}{\right}
\newcommand{\re}{\operatorname{Re}}
\newcommand\tot{\to0 \text{ as } t\to\infty}

\newcommand{\ap}{\mathcal{AP}}

\def\sgn{\operatorname{sgn}}

\newtheorem{Theorem}{Theorem}[section]
\newtheorem{Remark}[Theorem]{Remark}
\newtheorem{Lemma}[Theorem]{Lemma}

\newtheorem{Corollary}[Theorem]{Corollary}
\newtheorem{Proposition}[Theorem]{Proposition}
\newtheorem{Definition}[Theorem]{Definition}
\newtheorem{Example}[Theorem]{Example}

\numberwithin{equation}{section}

\begin{document}

	\title{On a Poincar\'e-Perron problem for high order differential equation}
	\date{\today}
	
	\author[H. Bustos]{H. Bustos}	
	\address{H. Bustos
	\newline\indent Centro de Docencia de Ciencias B\'asicas para Ingenier\'ia
	\newline\indent Facultad de Ciencias de la Ingenier\'ia
	\newline\indent Universidad Austral,
	\newline\indent	Campus Miraflores,
	\newline\indent Valdivia, Chile}
	\email{harold.bustos@uach.cl}
	
\author[P. Figueroa]{P. Figueroa}	
	\address{P. Figueroa
	\newline\indent Instituto de Ciencias F\'isicas y Matem\'aticas
	\newline\indent Facultad de Ciencias
\newline\indent Universidad Austral
\newline\indent Campus Isla Teja
\newline\indent Valdivia, Chile}
\email{pablo.figueroa@uach.cl}
	
	\author[M. Pinto]{M. Pinto.}	
	\address{M. Pinto
	\newline\indent Departamento de Matem\'aticas,
	\newline\indent Universidad de Chile,
\newline\indent	Las Palmeras 3425, Casilla 653, 
\newline\indent Santiago, Chile}
\email{pintoj@uchile.cl}

		\maketitle
	\bigskip
	\bigskip

	
	\begin{abstract}
		We address asymptotic formulae for the classical Poincar\'e-Perron problem of linear differential equations with almost constant coefficients in a half line $[t_0,+\infty)$ for high order equation $n\ge 5$ and some $t_0\in\mathbb{R}$. By using a scalar nonlinear differential equation of Riccati type of order $n-1$, we recover Poincar\'e's and Perron's results and provide asymptotic formulae with the aid of Bell's polynomials. Furthermore, we obtain some weaker versions of Levinson, Hartman-Wintner and Harris-Lutz type Theorems without the usual diagonalization process. For an arbitrary $n\ge 5$, these are corresponding versions to known results for cases $n=2,3$ and $4$.
	\end{abstract}
	
	\section{Introduction}

Higher order differential equations have been intensively and extensively studied during the last decades. They have been attracted great attention due to their possible applications to different areas such as theoretical physics, population dynamics, biology, ecology, etc., see for example \cite{F,Om2,Om3,Ol,PW1, PW2,PhPT}. Consider  $a_0,\dots,a_{n-1}\in\mathbb{C}$ and certain functions $r_0(t),\dots,r_{n-1}(t)$, a classical Poincar\'e-Perron problem of linear differential equations consists in
\begin{equation}\label{poinca}
	 y^{(n)}(t)+\sum_{i=0}^{n-1} (a_i+r_i(t))y^{(i)}(t)=0,\ \ \ \text{in $[t_0,+\infty)$}
\end{equation}
for some $t_0\in\mathbb{R}$, where functions $r_i$'s are small in some sense and $n\ge 2$. For simplicity, given a $(n+1)$-tuple $c=(c_0,\dots,c_{n})$ of complex numbers,  we denote $P(c;x):=\sum_{i=0}^{n}c_ix^{i}$, so that $P(a;x)=x^n+\sum_{i=0}^{n-1} a_i x^{i}$ is the characteristic polynomial of \eqref{poinca} with $r_i\equiv0$, $i=0,\dots,n-1$, where $a=(a_0,\dots ,a_{n-1}, 1)$, $a_i$'s are the complex numbers associated to the equation \eqref{poinca} and $a_n=1$. Precisely, for any $n\ge 2$, Poincar\'e \cite{Po} proved in 1885 the existence of a solution $y$ to \eqref{poinca} such that $y'(t)/y(t)$ converges as $t\to +\infty$, assuming that the roots $\lambda_1,\dots,\lambda_n$ of the polynomial $P(a;x)$ have distinct real part, $r_i$'s are continuous on $[t_0,+\infty)$ and $r_i(t)\to 0$ as $t\to +\infty$ for all $i=0,\dots,n-1$. At the beginning of the twentieth century, Perron \cite{Perr} improved this result, under the same hypothesis, assuring the existence of $n$ solutions $y_1, \dots, y_n$ with the same asymptotic behavior, namely, logarithmic derivative $y_i'(t)/y_i(t)$ converges to $\lambda_i$ as $t\to +\infty$. However, there is not a more precise formula for the asymptotic behavior of solutions in either case.

\medskip
Hartman in \cite[Th. 17.2]{Hartman} gives asymptotic estimates for a fundamental system of solutions to \eqref{poinca} if the functions $|r_i(t)|t^q$ are integrable on $[0,+\infty)$ for some $q\ge0$ and the real part of roots of $P(a;x)$ are distinct. Similar results have been obtained by \u Sim\u sa in \cite{S84,S85} under weaker assumptions on possibly not absolutely convergence of integrals of $r_i(t)t^q$. Trench in \cite{Trench86} proved a similar result to that in \cite{S85} under mild integral smallness condition. Later, \u Sim\u sa in \cite{S88} recovered and extended a Perron's result under weighted integrals involving $r_i$'s assumptions, obtaining estimates similar to \eqref{yjioyi} for a fixed $\lambda_i$, but without showing a more precise formula. Prevatt \cite{Prevat} using exponential dichotomy applies the results to spectral analysis and index theory which is very important in Schr\"odinger equation and similar ones. See also \cite[chapter 1]{Kiguradze} for related results.

\medskip
Another perturbations have been studied, for instance $r_i\in L^p$ for some $p\ge 1$, in systems of linear differential equations, which can be applied to this equation. These results are due to Levinson \cite[Th. 1.3.1]{East} for $p=1$, to Hartman-Wintner \cite[Th. 1.5.1]{East} for $p\in (1,2]$ and to Harris-Lutz \cite{Hl} for $p>2$.

\medskip

On the other hand, linear perturbations of the equation $y^{(n)}=0$, namely, equation \eqref{poinca} with $a_0=a_1=\cdots=a_{n-1}=0$, are studied for example in \cite{Nevai,Prevat} in the context of asymptotic integration. These kind of results depends on the fact that, the related linear system $Y'(t)=A(t)Y(t)$ possesses exponential dichotomy under some assumptions. Besides, in two works \cite{Trench76,Trench84}, Trench obtained the same results as Hartman \cite[Th. 17.1]{Hartman} under weaker assumptions (conditionally integrable): solutions behaves like polynomials. These results have been extended to the case $a_0=a_1=\cdots=a_{n-k-1}=0$ for some fixed $1\le k \le n-1$ in \cite{Trench88}; they extended also \cite[Th. 17.3]{Hartman}.

	 \medskip

	Problem \eqref{poinca} has been also studied for cases $n=2,\; 3$ and $4$ in several works \cite{CHP1,CHP2,FP1,FP2,FP3}, recovering Poincar\'e and Perron's results and providing asymptotic formulae. In order to study \eqref{poinca} a key step is to make the following change of variables $z=y^\prime/y-\lambda$\,, where $\lambda$ denotes a fixed root of $P(a;x)$. Then, the aim is to finding such a function $z$ depending on $r_i$'s with property $z(t)\tot$. This change is equivalent to consider
	\begin{equation}\label{yeilz}
	y(t)=\exp\left(\int_{t_0}^t [\lambda+z(s)]\,ds\right)\,, \quad t\ge t_0\,.
	\end{equation}
	Hence, replacing $y$ into equation \eqref{poinca} an equation for $z$ is obtained. For instance, if $n=2$ then $z$ satisfies a Riccati equation and asymptotic formulae for solutions to \eqref{poinca} were proved in \cite{FP1} . For $n\ge 3$ we say that $z$ satisfies a Riccati type equation and asymptotic formulae for solutions to \eqref{poinca} were proved in \cite{FP2,FP3} when $n=3$ and in \cite{CHP1,CHP2} when $n=4$.  Thus, inspired by these results our aim is to perform the same procedure for an arbitrary $n\ge 5$ and to obtain asymptotic formulae. More precisely, we will prove that $z^{(i)}(t)\tot$ for all $i=0,1,\dots,n-2$ and also, we will express a such $z$ as a sum of known terms. 
With the aid of complete Bell polynomials we will establish the equation for $z$ we will use along this work. Consider complete Bell polynomial $B_i$'s given by the following recursion formula $B_0=1$ and for $i\in \mathbb{N}$
\begin{equation}\label{complebell}
B_{i+1}(x_1,\dots,x_{i+1})=\sum_{j=0}^{i} {i \choose j}B_{i-j}(x_1,\dots,x_{i-j})x_{j+1}\,,
\end{equation}
These type of polynomials and related ones have been useful in combinatorial mathematics.  They also occur in many applications, such as in the Fa\'a di Bruno's formula, and also they are related to Stirling and Bell numbers, see \cite{AndBell,Bell}.

\begin{Theorem}\label{theo1}
Function $y$ given by \eqref{yeilz} is a solution to \eqref{poinca} in $[t_0,+\infty)$ if and only if $z$ is a solution to
\begin{equation}\label{dif}
\mathcal{ D}z(t)+P(r(t);\lambda)+\mathcal{L}(t,z)+\mathcal{ F}(t,z,\dots,z^{(n-2)})=0\,, \ \ \text{in $[t_0,+\infty)$,}
\end{equation}
where $\lambda$ is a fixed root of $P(a;x)$, the differential linear part with constant coefficients is the $n-1$ order differential operator
\begin{equation}\label{linearz}
	\mathcal{D}z = \sum_{j=1}^{n} \frac{1}{j!}\frac{\partial^{j}}{\partial x^j}P(a;\lambda)z^{(j-1)},
\end{equation}
the function depending only on perturbations $r=(r_0,\dots,r_n)$ with $r_n\equiv 0$ is
\begin{equation}\label{prl}
P(r(t);\lambda)=\sum_{k=0}^{n-1} \lambda^k r_k(t)=r_0(t)+\lambda r_1(t) + \lambda^2 r_2(t)+\cdots + \lambda^{n-1} r_{n-1}(t),
\end{equation}
the differential linear part with variable coefficients is
\begin{equation}\label{onedee}
\mathcal{L}(t,z):=\sum_{k=1}^{n-1} \frac{1}{k!}\frac{\partial^{k}}{\partial x^k}[P(r;\lambda)]z^{(k-1)}(t).
\end{equation}
and the nonlinear term is given by
\begin{equation}\label{formf}
	\mathcal{F}(t,z,z',\dots,z^{(n-2)})=\sum_{i=2}^{n}\sum_{j=0}^{n-i}{i+j \choose j}(a_{i+j}+r_{i+j}(t))\lambda^j\big(B_i(z,z,\dots,z^{(i-1)})-z^{(i-1)}\big)\,.
\end{equation}
\end{Theorem}

	\medskip 

As in \cite{Perr,Po}, we also assume that $P(a;x)$ has $n$ roots, $\lambda,\lambda_1,\dots,\lambda_{n-1}$, with different real parts, so that the characteristic polynomial of \eqref{linearz} $P_D(a,x)$ (see \eqref{defdpol}) will have roots with different and non zero real parts. Therefore an exponential dichotomy can be consider a scalar exponential dichotomy with scalar Green's functions. Given $\omega\in \C$ we define the function
\begin{equation}\label{defg}
g_{\omega}(t,s):=
\begin{cases}
-{\rm sgn}(\Re{\omega}) e^{\omega(t-s)} & {\rm sgn}(\Re{\omega})(t-s) < 0 \\
0 & \text{otherwise}
\end{cases} \,,
\end{equation}
where $\Re{\omega}$ denotes the real part of $\omega$. Hence, we  define Green's operators acting on locally integrable functions $f:[t_0,+\infty)\to\C$ given by
\begin{equation}\label{defls}
G_{\omega}\big[f\big](t):=\int_{t_0}^{\infty} g_{\omega}(t,s)f(s)\,ds\quad\text{and}\quad	I_{\omega}\big[f\big](t):=\int_{t_0}^{\infty}|g_{\omega}(t,s) f(s)|\,ds\,.
\end{equation}
These operators allows us to define the Green's function $\mathcal{G}$ and Green's operator $G$ for equation $\mathcal{D}z=f$ (see equations \eqref{greenfor} and \eqref{greenop}), so that an integral equation for $z$ is obtained, see \eqref{eiz}. Then, the following asymptotic formula as $t\to \infty$ is deduced
\begin{equation}\label{afintro}
    \begin{split}
        y_\lambda(t)=&\,(1+o(1)) e^{\lambda(t-t_0)}\\
        &\times \exp\left((-1)^n\prod_{j=1}^{n-1}(\lambda_j-\lambda)^{-1}\int_{t_0}^t [P(r;\lambda)(s)+\mathcal{L}(s,z(s))+\mathcal{F}(s,Z(s))]\,ds\right)\,.
    \end{split}
\end{equation}
Some conditions (see \eqref{cl}) insure the existence of $z$ satisfying \begin{equation}\label{estzi}
    z^{(i)}=O\big((I_{\beta}+I_{-\beta})[P(r;\lambda)]\big),
\end{equation}
where $0<\beta<\min\{\Re(\lambda_j - \lambda)\ :\ j=1,\dots,n-1\}$, concluding 
\begin{equation}\label{yjioyi}
    \frac{y_\lambda^{(i)}(t)}{y_\lambda(t)}=\lambda^{i} + O\big((I_{\beta}+I_{-\beta})[P(r;\lambda)]\big)
\end{equation}
for $i=1,\dots,n-1$ generalizing Poincare-Perron results \cite{Perr,Po}. The scalar treatment allows obtain sharp asymptotic formulas and error estimates. As Bencke \cite{Behncke,BehnckeII}  affirm, the error estimates are very important in the applications. All this scalar manipulations are translated in extensions of famous results of Levinson, Hartman-Wintner and Harris-Lutz. In particular, the following formula
$$y_\lambda(t)=\bigg[1+O\left(\int_{t}^{+\infty}(I_{\beta}+I_{-\beta})[P(r;\lambda)](s)\, ds \right)\bigg] e^{\lambda(t-t_0)}$$
has been shown in Theorem \ref{pplev} extending Levinson's Theorem \cite[Th. 1.3.1]{East}. Furthermore, we prove that
\begin{equation*}
    \begin{split}
        y_\lambda(t)=&\, \bigg[1+O\left(\int_{t}^{+\infty}(I_{\beta}+I_{-\beta})[R(\cdot,\theta)](s)\, ds \right)\bigg] e^{\lambda(t-t_0)}\\
        &\,\times \exp\bigg(- \sum_{j=1}^{n-1}{1\over \Gamma_j\gamma_j} G_{\gamma_j}\big[P(r;\lambda)\big](t)  +(-1)^n\prod_{k=1}^{n-1}(\lambda_k-\lambda)^{-1}\int_{t_0}^t P(r;\lambda)(s)\,ds\bigg),
    \end{split}
\end{equation*}
where $\theta=-G[P(r;\lambda)]$, $R(\cdot,\theta)=\mathcal{L}(\cdot,\theta) + \mathcal{F}(\cdot,\theta,\theta',\dots,\theta^{(n-2)})$, $\gamma_j=\lambda_j-\lambda$ and  $\Gamma_j=\prod_{k=1,k\ne i}^{n-1} (\lambda_i-\lambda_k)$, in Theorem \ref{tut} and in Theorem \ref{hllp} we show that
$$y_\lambda(t)=(1+o(1))e^{\lambda(t-t_0)}\exp\left(\int_{t_0}^t\left[ \theta_1(s)+\cdots + \theta_m(s)\right]\,ds\right),$$
where $\theta_1=\theta$ and $\theta_i$'s depends on $r_i$'s and they are recursively defined.

\medskip

The case $n=5$ is exhaustively studied and also a concrete equation of order 5 with diverse integrability conditions are examinated obtaining several nice formulae, with $r_1(t)=t^{-2/3}, r_2(t)=0, r_3(t)=(t^2+1)^{-1/3}$, $\lambda=1$ we obtain that 
$$
  y(t)= (1+O(1/t))ce^{\lambda(t-t_0)} \exp\left(\frac{\sqrt[3]{t}}{9} +\frac{t \sqrt{t^{2}+1}\left(2 t^{2}+5\right)}{48} \right)\left(t+\sqrt{t^2+1} \right)^{1/16}\,.
$$

Concerning possible generalizations, it would be interesting to know if similar results to Theorems \ref{orederzun} and \ref{base} are still valid in the class of almost periodic type functions. In other words, if it is possible to generalize Poincar\'e's and Perron's classical problem of approximation \eqref{poinca} to the class of almost periodic type functions independently of $n$, the order of the equation. The authors in \cite{FP4} have been addressed it for $n=2$. See \cite{BFP} for $n\ge 3$.

\medskip
The rest of the paper is organized as follows. Section \ref{io} is devoted to provide a review of useful results such as integral operators and complete Bell polynomials. Besides, we prove Theorem \ref{theo1} and we study a more general nonlinear differential equation. Main results are presented in section \ref{espp}. Precisely, in subsection 3.1 we have recovered Poincar\'e's and Perron's results and provide asymptotic formulae in both cases (see \eqref{fath1}  and \eqref{fath2}). Furthermore, Levinson and Hartman-Wintner type results and weaker version of them have been obtained under $L^p$-perturbations in subsection 3.2. We end section \ref{espp} with a Harris-Lutz type result, namely, $L^p$-perturbations with $p>2$. Our results improve, for a high order equation \eqref{poinca}, those results in \cite{East,Hl}, reducing the calculus in Theorem \ref{hllp} to $[p]$ summands ($[p]$ is the integer part of $p$). In \cite{East,Hl} the quantity of sumands in Theorem \ref{hllp} is $2^l$ (usually too big), where $2^{l-1}<p\le 2^l$
(see Theorem \ref{hllp} and \cite[Th. 2.8.3]{East} with $\nu=3$). In section 4, we discuss obtained results for $n=5$ and present an illustrative example. Finally, let us emphasize that the used method is scalar \cite{FP1,FP2,FP3}, reducing the order of the equation and avoiding the usual diagonalization process \cite{Co,East,EG1,EG2}.

	\section{Preliminaries}\label{io}

In this section we shall review some facts concerning Green's type operators, complete Bell polynomials, a Riccati type equation and we will show an existence result for a nonlinear equation, which will be applied to the Riccati type equation related with \eqref{poinca}. In order to make a better exposition of these topics, we consider four subsections.

\subsection{Integral operators}

Note that $G_{\omega}[f]$ solve the differential equation $z'-\omega z=f$ in $[t_0,+\infty)$, $I_\omega=I_{\Re\omega}$ and
\begin{equation}\label{cotagl}
	\big|G_\omega\big[f\big](t) \big|\leq  I_{\Re{\omega}}\big[f\big](t).
\end{equation}
In order to state our results, we will need the following facts.

\begin{Lemma}\label{oer}
Let $\gamma\in\C$ with $\Re{\gamma}\neq0$ and let $r$ be a locally
integrable function on $[t_0,\infty)$. Consider the functions
$$r^*(t)=\int_t^{t+1}|r(\tau)|\,d\tau\qquad\text{and}\qquad \overline r(t)=\sup_{s\ge t}(1+s-t)^{-1}\int_t^s|r(\tau)|\,d\tau.$$
Then the following statements are equivalent: ($i$) $r^*(t)\to 0$ as $t\to +\infty$, ($ii$) $\overline
r(t)\to 0$ as $t\to +\infty$ and ($iii$) $I_{\gamma}[r](t)\to 0$ as $t\to +\infty$. Furthermore,
$G_\gamma[r](t)\to 0$ as $t\to +\infty$ holds if $r$ is conditionally integrable in
$[t_0,\infty)$. If $r\in L^p[t_0,\infty)$ for some $p\ge 1$, then
$I_\gamma[r](t)\to 0$ as $t\to +\infty$ and $I_\gamma[r]\in
L^p[t_0,\infty)$.
\end{Lemma}


\begin{Lemma}[\cite{FP2}]\label{lxir}
    Let $\gamma\in\C$ with $\Re{\gamma}\neq0$ and $\xi:[t_0,+\infty)\to\R$ be a locally integrable function such that $I_\gamma[\xi]$ is bounded in $[t_0,+\infty)$. If $r(t)\to0$ as $t\to+\infty$ then $I_\gamma[\xi r](t)\to 0$ as $t\to+\infty$. 
\end{Lemma}

Operators $G_\omega$ and $I_\omega$ have been very useful in asymptotic integration, see for instance \cite{Bellm,CHP1,CHP2,East,FP2,FP3,FP4}. They also satisfy the following inequalities, see \cite{FP2} for a proof. Here, we denote ${\rm sgn}$ as the sign function.

\begin{Lemma}\label{techilem}
	Let  $\alpha$ be a non zero real number and  $0<\beta< |\alpha|$\,. Then  
	\begin{enumerate}
	    \item \begin{equation*}
	I_{\alpha}\big[aI_{\rm sgn(\alpha)\beta}[b]\big](t) \leq  I_{\rm sgn(\alpha)\beta}[b](t)I_{\alpha-\rm sgn(\alpha)\beta}[a](t) \,,
	\end{equation*}
	
	\item 
	$$
	I_{\alpha}[a](t)\leq I_{\rm sgn(\alpha)\beta}[a]\,,
	$$
	
	\item
	\begin{equation*}
	I_{\alpha}\big[a\big(I_{\beta}+I_{-\beta}\big)[b]\big)\big](t)  \leq
	2\big(I_{\beta}+I_{-\beta}\big)[b](t)I_{\alpha-{\rm sgn}(\alpha)\beta}[a](t)\,.
	\end{equation*}
	
	\end{enumerate}
	
\end{Lemma}

\subsection{Complete Bell polynomials}
Direct computations lead us to obtain the complete Bell polynomials using the recursion formula \eqref{complebell}. For example, it follows straightforward that the first 6 polynomials
\begin{equation}\label{bellpols}
	\begin{split}
	B_0 &=1 \\
	B_1(x_1) &=x_1 \\
	B_2(x_1,x_2)&= x_1^2+x_2 \\
	B_3(x_1,x_2,x_3)& = x_1^3+3x_1x_2+x_3 \\
	B_4(x_1,x_2,x_3,x_4)& = x_1^4+6x_1^2x_2+4x_1x_3+3x_2^2+x_4\\
	B_5(x_1,x_2,x_3,x_4,x_5)& = x_1^5+10x_1^3x_2+15x_2^2x_1 + 10x_1^2x_3+10x_3x_2+5x_4x_1+x_5
	\end{split}
\end{equation}
The complete Bell polynomials satisfy the following binomial type relation \cite{Bell} 
\begin{equation}\label{reducedz}
	B_i(x_1+y_1,\dots,x_i+y_i)=\sum_{j=0}^{i}{i \choose j} B_{i-j}(x_1,\dots,x_{i-j})B_j(y_1,\dots,y_j)\,.
\end{equation}
Also, by formula \eqref{complebell}, an induction shows that $B_j(\lambda,0,\dots,0)=\lambda^j $ for all $j\geq 1$\,. By the recursion formula \eqref{complebell}, it follows easily that the unique linear term of the complete Bell polynomial $B_i$ is $x_i$, $i\ge 1$. Moreover, define polynomials $f_i$, $i\ge 0$ by
\begin{equation}\label{bimzi}
    \begin{split}
        f_i(x_1,\dots,x_i)&=B_{i+1}(x_1,\dots,x_{i+1})-x_{i+1}=\sum_{j=0}^{i-1} {i\choose j} B_{i-j}(x_1,\dots,x_{i-j})x_{j+1},
    \end{split}
\end{equation}
in view of \eqref{complebell}. The following characterization will be useful to study \eqref{poinca} with perturbations $r_i$'s satisfying integrable conditions.

\begin{Lemma}
    Polynomials $f_i$, $i\ge 1$ can be expressed as
    \begin{equation}\label{dfi}
    f_i(x_1,\dots,x_i)=\sum_{k=2}^{i+1} h_{k,i}(x_1,\dots,x_i),
    \end{equation}
    where $h_{k,i}$ is a polynomial that contains only terms of degree $k$. 
\end{Lemma}

\begin{proof}
Note that $f_0=0$ and for $i\in\mathbb{N}$ polynomials $f_i$ can be found by using the recursion formula
\begin{equation}\label{fim1}
    \begin{split}
        f_{i+1}(x_1,\dots,x_{i+1})
        &=\sum_{j=0}^{i} {i+1\choose j} \left[f_{i-j}(x_1,\dots,x_{i-j})+x_{i-j+1}\right]x_{j+1}.
    \end{split}
\end{equation}
Direct computations shows that $f_1(x_1)=h_{2,1}(x_1)=x_1^2$, $f_2=h_{2,2}+h_{3,2}$ with $h_{2,2}(x_1,x_2)=3x_1x_2$, $h_{3,2}(x_1,x_2)=x_1^3$ and $f_3=h_{2,3}+h_{3,3}+h_{4,3}$, with $h_{2,3}(x_1,x_2,x_3)=4x_1x_3+3x_2^2$, $h_{3,3}(x_1,x_2,x_3)=6x_1^2x_2$ and $h_{4,3}(x_1,x_2,x_3)=x_1^4$
so that, \eqref{dfi} is true for $i=1,2,3$. By induction assume that \eqref{dfi} is true for $i=l$. So, from \eqref{fim1} for $i=l+1$ we have that
\begin{equation*}
        f_{l+1}(x_1,\dots,x_{l+1})
        =\sum_{j=0}^{l-1} {l+1\choose j} f_{l-j}(x_1,\dots,x_{l-j})x_{j+1} +\sum_{j=0}^{l} {l+1\choose j} x_{l-j+1}x_{j+1},
\end{equation*}
in view of $f_0=0$. Hence, we deduce that for $l\ge 0$
$$h_{2,l+1}(x_1,\dots,x_{l+1})= \sum_{j=0}^{l} {l+1\choose j} x_{l-j+1}x_{j+1},$$
is a polynomial that contains only terms of degree 2. Replacing \eqref{dfi} with $i=l-j$ and changing the order in the sum, we find that
\begin{equation*}
        \sum_{j=0}^{l-1} {l+1\choose j} f_{l-j}(x_1,\dots,x_{l-j})x_{j+1}
        =\sum_{k=2}^{l+1} \sum_{j=0}^{l-k+1} {l+1\choose j} h_{k,l-j}(x_1,\dots,x_{l-j})x_{j+1}
\end{equation*}
and $\displaystyle \sum_{j=0}^{l-k+1} {l+1\choose j} h_{k,l-j}(x_1,\dots,x_{l-j})x_{j+1}$ is a polynomial that contains terms of degree $k+1$. Therefore, defining for $k\ge 3$
$$h_{k,l+1}(x_1,\dots,x_{l+1})= \sum_{j=0}^{l-k+2} {l+1\choose j} h_{k-1,l-j}(x_1,\dots, x_{l-j+})x_{j+1},\quad l\ge 0,$$
the conclusion follows.
\end{proof}

\medskip

\subsection{Riccati type equation}	
	
	\medskip

\begin{proof}[Proof of Theorem \ref{theo1}]
The proof will be done in several steps. We compute $y^{(i)}$ in terms of $z$\,, for $i=0,1,2,\dots,n$\,. To do that, we first fix the notation setting
	$$
	y(t)=e^{g(t)}\,,\quad \mbox{ where } g(t)=\int_{t_0}^t(\lambda+z(s))\,ds\,,
	$$
so that for $i=0,\dots,n$ we have that
$$y^{(i)}=\frac{d^{i}}{dt^{i}}e^{g} \quad \text{ and }\quad 	\frac{y^{(i)}}{y} = e^{-g}\frac{d^{i}}{dt^{i}}e^{g}.$$
Then, denoting by $g_i(t)=g^{(i)}(t)$, $i=0,\dots,n$ and using the Leibniz rule, it follows that
\begin{equation}\label{dlr}
\frac{y^{(i+1)}}{y}=e^{-g}\frac{d^{i}}{dt^{i}}(g_1e^g)= \sum_{j=0}^{i}{i \choose j}e^{-g}\frac{d^{i-j}}{dt^{i-j}}(e^g)\frac{d^{j}}{dt^{j}} g_1= \sum_{j=0}^{i}{i \choose j} \frac{y^{(i-j)}}{y} g_{j+1}.
\end{equation}
Thus, taking into account \eqref{complebell} and \eqref{dlr}, we find that
$$\frac{y^{(i)}}{y}=B_i(g_1,\dots,g_i)= e^{-g}\frac{d^i}{dt^i}e^{g}\,
	$$
then, we verify that
$$\frac{y^{(i+1)}}{y}= B_{i+1}(g_1,\dots,g_{i+1})= \sum_{k=0}^{i}{i \choose j} B_{i-j}(g_1,\dots,g_{i-j})g_{j+1} \,.$$
for $i\in \mathbb{N}$. Returning to our problem, we have that
	\begin{equation}\label{ynb}
	\frac{y^{(i)}(t)}{y(t)}= B_{i}\big(\lambda+z(t),z'(t),\dots,z^{(i-1)}(t)\big)\,,
	\end{equation}
so that, replacing in the equation \eqref{poinca} we get
	\begin{equation}\label{poincaz}
	\sum_{i=0}^{n}(a_i+r_i(t))B_i\big(\lambda+z(t),z'(t),\dots,z^{(i-1)}(t)\big)=0\,,
	\end{equation}
with $a_n=1$ and $r_n(t)\equiv 0$. The Eq. $\!$\eqref{poincaz} is a nonlinear differential equation of order $n-1$. Hence, from \eqref{reducedz} we have that 
\begin{equation}\label{bilz}
    B_i\big(\lambda+z(t),z'(t),\dots,z^{(i-1)}(t)\big)=\sum_{j=0}^{i}{i \choose j} \lambda^{i-j}B_{j}(z(t),\dots,z^{(j-1)}(t))\,.
\end{equation}
Therefore, replacing \eqref{bilz} in \eqref{poincaz} 
and changing the order of the sums, we can rewrite the equation \eqref{poincaz} as
		\begin{equation}\label{poincazbell}
			\sum_{i=0}^{n}\sum_{j=0}^{n-i}{i+j \choose j} (a_{i+j}+r_{i+j}(t))\lambda^jB_{i}\big(z(t),z'(t),\dots,z^{(i-1)}(t)\big)=0\,.
		\end{equation}
It follows that equation \eqref{poincaz} has an independent part of $z$ and its  derivatives given by \eqref{prl}, 
since $B_0=1$ and $\lambda$ is a root of $P(a;x)$. Concerning the linear part of \eqref{poincaz} we have that 
by the recursion formula \eqref{complebell}, it follows easily that the unique linear term of the complete Bell polynomial $B_i$ is $x_i$\,. Hence,  the differential linear part with constant coefficients of equation \eqref{poincaz} is 
$$
\sum_{i=1}^{n}\sum_{j=1}^{i}{i \choose j} a_i\lambda^{i-j} z^{(j-1)}(t)  = \sum_{j=1}^{n} \sum_{i=j}^{n} {i \choose j} a_i\lambda^{i-j} z^{(j-1)}(t).
$$
Equality \eqref{linearz} follows just comparing the previous expression with the well known formula
$$
\frac{1}{j!}\frac{\partial^{j}}{\partial x^j}P(a;\lambda)=\sum_{i=j}^{n} a_i{i \choose j} \lambda^{i-j}\,.
$$

Similarly, note that equation \eqref{poincaz} also contains a variable linear part given by \eqref{onedee}, in view of $r_n\equiv 0$, so that $\dfrac{\partial^n}{\partial x^n}P(r;\lambda)\equiv 0$.

To end the proof of Theorem \ref{theo1}, we now focus in the non linear part of Eq. \eqref{poincaz}. Taking into account the linear differential part of \eqref{poincaz} and \eqref{onedee} follows  that the non linear part of equation \eqref{poincaz} corresponds to  
\begin{equation}\label{deff}
\begin{split}
\mathcal{ F}(t,z,\dots,z^{(n-2)}) := & \sum_{i=0}^{n}(a_i+r_i(t))B_{i}\big((\lambda+z(t)),z'(t),\dots,z^{(i-1)}(t)\big) \\
& \hspace{2cm} - \mathcal{ D}z(t) -P(r(t);\lambda)- \mathcal{ L}(t,z).
\end{split}	
\end{equation}
Note that, $\mathcal{F}$ is a polynomial in the derivatives of $z$ with monomials of degree bigger than 1. Moreover, by the equation \eqref{poincaz} it follows that $\mathcal{ F}$ is a polynomial of degree $n$. In fact, one has the following result.
\begin{Lemma}\label{nonlinearp}
	The non linear part of equation \eqref{poincaz}, i. e., the map $\mathcal{ F}$ defined in eq. \eqref{deff} can be written by \eqref{formf}. Moreover, $	\mathcal{F}$ does not depend of $z^{(n-1)}$\,, i.e., is a polynomial in $n-1$ variables.
\end{Lemma}
\begin{proof}
	It follows from binomial type relation \eqref{reducedz} that the equation \eqref{poincaz} is equivalent to \eqref{poincazbell}. The result follows from \eqref{linearz}-\eqref{onedee} and noticing that, the polynomials $B_i(z,z,\dots,z^{(i-1)}))-z^{(i-1)}$ only contains monomials of degree bigger than 1, for all  $i\geq2$\,. 
\end{proof}

Note that, 
since $a_n=1$ and $r_n=0$, one can see that the monomial in $\mathcal{ F}$ of degree $n$ is $z^n$. 	

\medskip
Therefore, taking into account \eqref{linearz},  \eqref{prl}, \eqref{onedee} and Lemma \ref{nonlinearp}, equation \eqref{poinca}, i.e. \eqref{poincaz} can be rewritten as \eqref{dif}. This finishes the proof of Theorem \ref{theo1}
\end{proof}

In order to study Eq. \eqref{dif}, it will be useful to understand operator $\mathcal{D}$.

\begin{Lemma}\label{dpol}
	We denote $P_\mathcal{D}(a;x)$ the characteristic polynomial associated to the linear differential operator $\mathcal{D}$, i.e.,
	\begin{equation}\label{defdpol}
		P_\mathcal{D}(a;x):=\sum_{k=1}^{n} \frac{1}{k!}\frac{\partial^{k}}{\partial x^k}P(a;\lambda)x^{k-1}\,.
	\end{equation}
	Suppose that $\{\lambda\}\cup\{\lambda_l\mid l=1,\dots n-1\}$ are the roots of $P(a;x)$. Then $\lambda_l-\lambda$\,,	for $l=1,\dots n-1$\, are the roots of  $P_{\mathcal{D}}(a;x)$.
\end{Lemma}
\begin{proof}
 Note that 	if $\lambda\not=\lambda_l$ for $l=1,\dots n-1$\,. We see that for each $l$
 $$
 0=P(a,\lambda_l)-P(a,\lambda)=\sum_{i=0}^{n}a_i (\lambda_l^i-\lambda^i)= \sum_{i=1}^{n}a_i (\lambda_l^i-\lambda^i)\,.
 $$
 On the other hand, for each $i>0$ we have that 
 \begin{equation*}
 	\begin{split}
 	\lambda_l^i-\lambda^i &= \big((\lambda_l-\lambda)+\lambda\big)^i-\lambda^i 
 	= \sum_{k=0}^{i} {i \choose k} \lambda^{i-k}(\lambda_l-\lambda)^k-\lambda^i= \sum_{k=1}^{i} {i \choose k} \lambda^{i-k}(\lambda_l-\lambda)^k \,.
 	\end{split}
 	\end{equation*}
 Thus, replacing it and using that $P(a;\lambda)=0$, we obtain 
 \begin{equation*}
 \begin{split}
 0
 &= \sum_{i=1}^{n}a_i \sum_{k=1}^{i} {i \choose k} \lambda^{i-k}(\lambda_l-\lambda)^k\,.
\end{split}
\end{equation*}
Finally, dividing the previous equality by $\lambda_l-\lambda$\,, we obtain
 $$
 0=\sum_{i=1}^{n}a_i \sum_{k=1}^{i} {i \choose k} \lambda^{i-k}(\lambda_l-\lambda)^{k-1} = \sum_{k=1}^{n} \frac{1}{k!}\frac{\partial^{k}}{\partial x^k}P(a;\lambda)(\lambda_l-\lambda)^{k-1}\,.
 $$
\end{proof}

\begin{Corollary}\label{difertroots}
	If $P(a,x)$ has roots with different real parts, then the polynomial $P_\mathcal{ D}(a,x)$ has roots with different and non zero real parts.
\end{Corollary}

\begin{Remark}
	{\rm 
	We shall study Eq. \eqref{poinca}, i.e., Eq. \eqref{poincaz}, by addressing Eq. \eqref{dif}. Taking $b(t)=-P(r(t);\lambda)$ and $R(t,z,\dots,z^{(n-1)})=-\mathcal{L}(t,z)-\mathcal{ F}(t,z,\dots,z^{(n-2)})\,$, a more general problem will be useful.
}
\end{Remark}

\subsection{A more general nonlinear problem}

Hence, we will study, in a general way, equations of the form
\begin{equation}\label{generaleq}
	\mathcal{ D}z(t)=b(t)+R(t,z(t),\dots,z^{(n-2)}(t))\,,
\end{equation}
where $\mathcal{D}$ is a differential operator of order $n-1$ with constant coefficients, the characteristic polynomial, associated to differential operator $\mathcal{ D}$ has $n-1$ roots with non zero real part. Precisely, assume that $\{\gamma_m\}_{m=1}^{n-1}$ are the $n-1$ different roots of the characteristic polynomial of $\mathcal{ D}$\,. We also denote $\alpha_m:=\Re(\gamma_m)$\,, where $\Re$ denotes the real part. Also, we shall assume that $b$ is a locally integrable function and $R:\R\times \C^{n-1}:\to\C$, with $R(\cdot,Z)$ locally integrable for $Z\in\C^{n-1}$. 

\medskip

The method of the variation of constants lead us to define the Green's function of the homogeneous equation $\mathcal{D}z=0$. First, we denote by
\begin{equation}\label{Gammaj}
\Gamma_{i}:= \prod_{k= 1,\,k\not=i}^{n-1}(\gamma_i-\gamma_k).
\end{equation}
 Thus, we define the Green's function $\mathcal{G}(t,s)$ by 
\begin{equation}\label{greenfor}
\mathcal{G}(t,s)= \sum_{i=1}^{n-1}\frac{1}{\Gamma_j}g_{\gamma_i}(t,s),
\end{equation}
where $g_\gamma$ is given by \eqref{defg}. Therefore, our Green operator for equation $\mathcal{D}z=f$ is given by
\begin{equation}\label{greenop}
    G\big[f\big](t):=\int_{t_0}^{\infty} \mathcal{G}(s,t)f(s)\,ds = \sum_{j=1}^{n-1}\frac{1}{\Gamma_j}G_{\gamma_j}\big[f\big](t)\,.
\end{equation}
Notice that for any $i=1,\dots,n-2$ we have that
\begin{equation}\label{derivategi}
	G\big[f\big]^{(i)}(t)= \sum_{j=1}^{n-1}\frac{\gamma_j^i}{\Gamma_j}	G_{\gamma_j}\big[f\big](t)
\end{equation}
and
\begin{equation}\label{derivategnm1}
	G\big[f\big]^{(n-1)}(t)= \sum_{j=1}^{n-1}\frac{\gamma_j^{n-1}}{\Gamma_j}	G_{\gamma_j}\big[f\big](t) +f(t),
\end{equation}
so that $\mathcal{D}(G[f])=f,$ since
$$\sum_{j=1}^{n-1}\frac{\gamma_j^i}{\Gamma_j}=\begin{cases}
0& i=0,\dots,n-2\\
1&i=n-1
\end{cases}.
$$

\begin{Remark}\label{cotadervG}
	{\rm 
		Note that from \eqref{derivategi} one can to estimate for a locally integrable function $f$ the value of $|G\big[f\big]^{(i)}(t)|$ in the following way
		\begin{equation}\label{cotalsup}
		|G\big[f\big]^{(i)} (t)| \leq \sum_{j=1}^{n-1}\left|\frac{\gamma_j^{i}}{ \Gamma_j}G_{\gamma_j}[f](t) \right| \leq	\sum_{j=1}^{n-1}\frac{|\gamma_j|^{i}}{|\Gamma_j|}I_{\Re\gamma_j}[f](t)\leq \|f\|_\infty \sum_{j=1}^{n-1}\frac{|\gamma_j|^{i}}{|\Gamma_j|} I_{\Re\gamma_j}[1](t)\,.
		\end{equation}
	}
\end{Remark}

\medskip
Given $Z\in \C^{n-1}$, with $Z=(z_0,\dots,z_{n-2})$ we consider the norm $\|Z\|=\sum_{i=0}^{n-2}|z_i|$. Returning to Eq. \eqref{generaleq}, we shall assume that there exists a constant $M>0$ and a locally integrable bounded function $\xi_M:[t_0,+\infty)\to [0,+\infty)$ such that for all $t\ge t_0$ and $Z_i\in\C^{n-2}$ with $\|Z_i\|\leq M$\,, for $i=1,2$
\begin{equation}\label{lipchitz}
    |R(t,Z_1)-R(t,Z_2)|\leq \xi_M(t)\|Z_1-Z_2\|.
\end{equation}
Now, given a function $z:[t_0,+\infty)\to \mathbb{C}$, for simplicity we shall denote  $Z=Z(t)$ the $(n-1)$-tuple $Z=(z,z',\dots,z^{(n-2)})$ and consider the function space $\mathcal{C}_0^{n-2}=\mathcal{C}_0^{n-2}[t_0,+\infty)$ defined by a function $z:[t_0,+\infty)\to \mathbb{C}$ belongs to $ \mathcal{C}_0^{n-2}$ if and only if $z^{(i)}$ is continuous and $z^{(i)}(t)\to 0$ as $t\to+\infty$ for $i=0,\dots,n-2$. With some abuse of notation we denote $\|Z(t)\|=\sum_{i=0}^{n-2}|z^{(i)}(t)|$\,. Observe that $\mathcal{C}_0^{n-2}$ is a Banach space with norm $\|z\|_{\mathcal{C}_0^{n-2}}=\sup_{t\in[t_0,+\infty)}\|Z(t)\|$.

For the next result, we denote $\displaystyle \tilde\alpha_j=\sum_{i=0}^{n-2}|\gamma_j^{i}|>0$ for $j=1,\dots,n-1$ and $\displaystyle\tilde{\gamma}=\sum_{k=1}^{n-1}\frac{\tilde{\alpha}_k}{\left|\Gamma_k\right|}$. Recall that $\alpha_j=\Re\gamma_j$ for $j=1,\dots,n-1$.

\begin{Proposition}\label{fixord}
Suppose that $\displaystyle\sum_{i=0}^{n-2}\left|G[b]^{(i)}(t)\right|\to 0$ as $t\to +\infty$. 
Moreover, assume that on $[t_0,\infty)$ $R(\cdot,0)=0$ and $R$ satisfies \eqref{lipchitz} for some constant $M>0$. If there exists a positive constant $0<\epsilon_0< 1$ such that 
\begin{equation}\label{intgcond0}
	\Big\| \sum_{j=1}^{n-1}\frac{\tilde\alpha_j}{\left|\Gamma_j\right|} I_{\alpha_j}\big[\xi_M\big] \Big\|_\infty \leq \epsilon_0 \,, \quad\ \, 
	\left\| \sum_{i=0}^{n-2}\left|G[b]^{(i)}(t)\right|\right\|_\infty \leq(1-\epsilon_0)M\,,
\end{equation}
where $\Gamma_j$ is defined in Eq. \eqref{Gammaj}, then equation \eqref{generaleq} has unique solution $z$ on $[t_0,\infty)$ such that $z^{(i)}(t)\to 0$ as $t\to+\infty$ for all $i=0,1,\dots,n-2$ and satisfies the integral equation
$$
z=G\big[b+ R(\cdot,z,z',\dots, z^{(n-2)})\big].
$$
If, in addition, there exists $0<\beta<\min\{|\alpha_j| \mid 1\leq j\leq n-1  \}$ such that 
\begin{equation}\label{intgcond}
	\Big\| \sum_{j=1}^{n-1}\frac{\tilde\alpha_j}{\left|\Gamma_j\right|} I_{\alpha_j-{\rm sgn}(\alpha_j)\beta}\big[\xi_M\big] \Big\|_\infty < \frac{1}{2},
\end{equation}
then $z^{(i)}=O\big( (I_{\beta}+I_{-\beta})[b] \big)$ for all $i=0,1,\dots,n-2$\,.
\end{Proposition}

\begin{proof} 
	For functions $z\in \mathcal{C}_0^{n-2}$ define the operator 
	$$
	Tz(t)= \int_{t_0}^{+\infty} \mathcal{G}(t,s)\big[b(s)+R(s, Z(s))\big]\,ds =G\big[b+R(\cdot,Z)\big](t)\,,
	$$
	where $Z=(z,z',\dots,z^{(n-2)})$. Notice that from the definition of Green's operator $G$ and its derivatives \eqref{derivategi} it follows that $(Tz)^{(i)}$ is a continuous function for $i=0,\dots,n-2$. Set
	$B_M=\big\{z\in\mathcal{C}^{n-2}_0 \, :\, \|z\|_{\mathcal{C}_0^{n-2}} \leq M \,\big\}$.
	Note that $B_M$ is a closed subset of $\mathcal{C}^{(n-2)}_0$. Also, if $z\in B_M$\,, we have that $\|z^{(i)}\|_\infty\leq M$ for all $i=0,\dots,n-2$. Hence, it follows that $R$ satisfies \eqref{lipchitz} for $z_1,\,z_2\in B_M$. Thanks to Lemma \ref{techilem}, $I_{\alpha_j}\big[\xi_M\|Z\|\big](t)\to 0$ as $t\to+\infty$ for  $z\in B_M$. Also, by using $R(\cdot,0)=0$, we have that for  $z\in B_M$
	\begin{equation}\label{grz}
	    \big|G^{(i)}[R(\cdot,Z)](t)\big|	\leq  \sum_{j=1}^{n-1} \left|\frac{\gamma_j^i}{\Gamma_j}\right| I_{\alpha_j}\big[R(\cdot,Z)](t)  
	\leq  \sum_{j=1}^{n-1} \left|\frac{\gamma_j^i}{\Gamma_j}\right| I_{\alpha_j}\big[\xi_M\|Z\|\big](t)
	\end{equation}
so that, $G^{(i)}[R(\cdot,Z)](t)\to 0 $ as $t\to+\infty$. Furthermore, we have that
$$	\big|(Tz)^{(i)}(t)\big| \leq 	\big|G[b]^{(i)}(t)\big|+\big|G[R(\cdot,Z)]^{(i)}(t)\big|.$$
Thus, from Lemma \ref{lxir} and conditions on $G[b]^{(i)}$'s we deduce that for all $i=0,\dots, n-2$ it holds $(Tz)^{(i)}(t)\to 0$ as $t\to+\infty$, namely, $Tz\in\mathcal{C}_0^{n-2}$.

On the other hand, we have that $I_{\alpha_j}\big[\xi_M\|Z\|\big](t)\leq M I_{\alpha_j}\big[\xi_M \big](t)$, so that
\begin{equation*}
	\begin{split}
	\big|(Tz)^{(i)}(t)\big|
	\leq & \big|G[b]^{(i)}(t)\big|+M \sum_{j=1}^{n-1} \left|\frac{\gamma_j^i}{\Gamma_j}\right| I_{\alpha_j}[\xi_M ](t).
		\end{split}
	\end{equation*}
 Note that changing order in the sum and using the choice of $\tilde{\alpha}_j$'s we get
$$\sum_{i=0}^{n-2} \sum_{j=1}^{n-1} \Big|\frac{\gamma_j^{i}}{\Gamma_j}\Big|=\sum_{j=1}^{n-1} \frac{1}{|\Gamma_j|} \sum_{i=0}^{n-2} |\gamma_j^{i}|=\sum_{j=1}^{n-1} \frac{\tilde\alpha_j}{|\Gamma_j|}=\tilde{\gamma}$$
and similarly as above,
$$\sum_{i=0}^{n-2} \sum_{j=1}^{n-1} \Big|\frac{\gamma_j^{i}}{\Gamma_j}\Big|I_{\alpha_j}[\xi_M ] 
=\sum_{j=1}^{n-1} \frac{\tilde\alpha_j}{|\Gamma_j|}I_{\alpha_j}[\xi_M ].$$ 
Hence, from \eqref{intgcond} we obtain that
\begin{equation*}
		\begin{split}
	\sum_{i=0}^{n-2}\big|(Tz)^{(i)}(t)\big| 
	\leq& \sum_{i=0}^{n-2}\bigg(\big|G[b]^{(i)}(t)\big|+M \sum_{j=1}^{n-1}\left|\frac{\gamma_j^i}{\Gamma_j}\right| I_{\alpha_j}[\xi_M ](t)\bigg)\\
	\le & \sum_{i=0}^{n-2}\big|G[b]^{(i)}(t)\big|+M \sum_{j=1}^{n-1}\frac{\tilde\alpha_j}{\left|\Gamma_j\right|} I_{\alpha_j}[\xi_M ](t).
		\end{split}
	\end{equation*}
By using assumptions \eqref{intgcond0}, it follows that $\|Tz\|_{\mathcal{C}_0^{n-2}}\le M$. Therefore, we get that $T(B_M)\subseteq B_M$.
	
	In order to show that $T$ has a fixed point we estimate similarly as above. Then, for  $z_1,\,z_2\in B_M$ and denoting $Z_i=(z_i,z_i',\dots,z_i^{(n-2)})$, $i=1,2$, we get that 
	\begin{equation*}
	    \begin{split}
\sum_{i=0}^{n-2}\big| Tz_1^{(i)}(t)- Tz_2^{(i)}(t) \big|
	\leq & \sum_{i=0}^{n-2}\sum_{j=1}^{n-1} \left|\frac{\gamma_j^i}{\Gamma_j}\right| I_{\alpha_j}\big[R(\cdot,Z_1(\cdot))-R(\cdot,Z_2(\cdot))\big](t)  \\
	\leq & \sum_{i=0}^{n-2}\sum_{j=1}^{n-1} \left|\frac{\gamma_j^i}{\Gamma_j}\right| I_{\alpha_j}\big[\xi_M\|Z_1(\cdot)-Z_2(\cdot)\|\big](t)  \\
	\leq & \,\|z_1-z_2\|_{\mathcal{C}_0^{n-2} } \sum_{i=0}^{n-2} \sum_{j=1}^{n-1} \left|\frac{\gamma_j^i}{\Gamma_j}\right| \, I_{\alpha_j}\Big[\xi_M \Big](t)
	  \end{split}
	\end{equation*}
Thus, by using \eqref{intgcond0} we conclude that $T$ is a contractive operator, since $\|Tz_1-Tz_2\|_{\mathcal{C}_0^{n-2} } \leq \epsilon_0 \|z_1-z_2\|_{\mathcal{C}_0^{n-2}}$ and $0<\epsilon_0<1$. Therefore, there exists a unique fixed point $z\in B_M$\,. 

\medskip
Now, consider the sequence $\{z_n\}\subseteq \mathcal{C}^{n-2}_0$ given by $z_0=0$ and for $n\ge 0$
	$$
	z_{n+1}(t)=Tz_n(t)= \int_{t_0}^{+\infty} \mathcal{G}(t,s)\big[b(s)+R(s, Z_n(s))\big]\,ds =G\big[b+R(\cdot,Z_n)\big](t)\,,
	$$
	where $Z_n=(z_n,z_n',\dots,z_n^{(n-2)})$. Since $T$ is a contractive operator $z_n\to z$ as $n\to +\infty$ in $\mathcal{C}^{n-2}_0$. So, we shall prove that for all $t\ge t_0$
	\begin{equation}\label{estzni}
	   \|Z_{n+1}(t)\|=\big\|\sum_{i=0}^{n-2}(Tz_n)^{(i)}(t)\big\| \leq \tilde{\gamma} N \big|\big(I_{\beta}+I_{-\beta}\big)[b](t)\big|,
	\end{equation}
	with a positive constant $N$ given by $N(1-2K)=1$ and $K$ satisfying
	$$\Big\| \sum_{j=1}^{n-1}\frac{\tilde\alpha_j}{\left|\Gamma_j\right|} I_{\alpha_j-\rm sgn(\alpha_j)\beta}\big[\xi_M\big] \Big\|_\infty \leq K<\frac{1}{ 2}.$$
	From \eqref{defls} and \eqref{cotagl} it follows that $|G_{\gamma_j}[b](t)|\leq I_{\rm sgn (\alpha_j)\beta}[b](t)\le (I_\beta + I_{-\beta})[b](t),$ so that, inequality \eqref{estzni} is clearly true for $n=0,1$. Indeed,
	$$|(Tz_1)^{(i)}(t)|=|G[b]^{(i)}(t)|\leq \sum_{j=1}^{n-1}{|\gamma_j^{i}|\over |\Gamma_j|} I_{\alpha_j}[b](t)\leq \sum_{j=1}^{n-1}{|\gamma_j^{i}|\over |\Gamma_j|} (I_\beta +I_{-\beta})[b](t) $$
so that	
$$\sum_{i=0}^{n-2}|(Tz_1)^{(i)}(t)| \leq \sum_{j=1}^{n-1}{\tilde\alpha_j\over |\Gamma_j|} (I_\beta +I_{-\beta})[b](t) $$
Suppose that \eqref{estzni} is true for $n=k$. So, for $n=k+1$ we use \eqref{grz} with $Z=Z_{k+1}$ and obtain that
\begin{equation*}
    \begin{split}
        \big|G^{(i)}[R(\cdot,Z_{k})](t)\big|
	&\leq \tilde{\gamma} N \sum_{j=1}^{n-1} \left|\frac{\gamma_j^i}{\Gamma_j}\right|  I_{\alpha_j}\Big[\xi_M \big( I_{\beta}+I_{-\beta}\big)[b]\Big](t)\\
	&\leq 2\, \tilde{\gamma}\, N \big( I_{\beta}+I_{-\beta}\big)[b](t)\,  \sum_{j=1}^{n-1} \left|\frac{\gamma_j^i}{\Gamma_j}\right| I_{\alpha_j-\rm sgn(\alpha_j)\beta}[\xi_M ](t),
    \end{split}
\end{equation*}
since from Lemma \ref{techilem} we have that 
$$
I_{\alpha_j}\Big[\xi_M \big( I_{\beta}+I_{-\beta}\big)[b]\Big](t) \leq 2 I_{\alpha_j-\rm sgn(\alpha_j)\beta}\big[\xi_M\big] (t) \,  \big(I_{\beta}+I_{-\beta}\big)[b]\big](t).
$$
Then, we obtain that
\begin{equation*}
	\begin{split}
	\big|(Tz_{k})^{(i)}(t)\big| 
	\leq & \sum_{j=1}^{n-1} \Big|\frac{\gamma_j^{i}}{\Gamma_j}G_{\gamma_j}[b](t)\Big|+2\,\tilde{\gamma}\,N \big( I_{\beta}+I_{-\beta}\big)[b](t)\,  \sum_{j=1}^{n-1} \left|\frac{\gamma_j^i}{\Gamma_j}\right| I_{\alpha_j-\rm sgn(\alpha_j)\beta}[\xi_M ](t)\\
	\leq& (I_\beta + I_{-\beta})[b](t) \left[\sum_{j=1}^{n-1} \Big|\frac{\gamma_j^{i}}{\Gamma_j}\Big|+2\,\tilde{\gamma}\,N \sum_{j=1}^{n-1} \left|\frac{\gamma_j^i}{\Gamma_j}\right| I_{\alpha_j-\rm sgn(\alpha_j)\beta}[\xi_M ](t) \right].
		\end{split}
	\end{equation*}
Note that from the choice of $\tilde{\alpha}_j$'s we get
$$\sum_{i=0}^{n-2} \sum_{j=1}^{n-1} \Big|\frac{\gamma_j^{i}}{\Gamma_j}\Big|I_{\alpha_j-\rm sgn(\alpha_j)\beta}[\xi_M ] 
=\sum_{j=1}^{n-1} \frac{\tilde\alpha_j}{|\Gamma_j|}I_{\alpha_j-\rm sgn(\alpha_j)\beta}[\xi_M ].$$ 
Hence, from \eqref{intgcond} we obtain that
\begin{equation*}
	\begin{split}
	\sum_{i=0}^{n-2}\big|(Tz_{k})^{(i)}(t)\big| 
	\leq& \tilde{\gamma}\,(I_\beta + I_{-\beta})[b](t)\left[1+2N \sum_{j=1}^{n-1} \frac{\tilde\alpha_j}{|\Gamma_j|} I_{\alpha_j-\rm sgn(\alpha_j)\beta}[\xi_M ](t)\right]\\
	\le & \tilde{\gamma}\,\big( I_{\beta}+I_{-\beta}\big)[b](t)[1+2KN].
		\end{split}
	\end{equation*}
Since $N=1+2KN$, we obtain that \eqref{estzni} holds for $n=k+1$. Therefore, it holds that $z^{(i)}=O\left((I_\beta+I_{-\beta})[b]\right)$ for all $i=0,\dots, n-2$. This concludes the proof.
\end{proof}

Note that $\displaystyle\sum_{i=0}^{n-2}\left|G[b]^{(i)}(t)\right|\to 0$ as $t\to +\infty$ if and only if  $G_{\gamma_j}[b](t)\to 0$ as $t\to +\infty$ for all $j=1,\dots,n-1$, in view of \eqref{derivategi} and the invertibility of the Vandermonde matrix $V=(a_{ij})_{i,j=1}^{n-1}$ with $a_{ij}=\gamma_{j}^{i-1}$. Also, observe that condition \eqref{intgcond}  implies the first inequality of \eqref{intgcond0} by using Lemma \ref{techilem}.

\bigskip
\section{Existence of solutions for the Poincar\'e Problem}\label{espp}
We return to study Eq. \eqref{dif}, where $\mathcal{ D}$ is defined \eqref{linearz}, $P(r;\lambda)$ and $\mathcal{ L}$ are defined by \eqref{prl} and \eqref{onedee}, respectively, and $\mathcal{ F}$ is the non linear part of the equation given by \eqref{formf}.

Note that given $\omega\in\mathbb{C}$, $\omega\ne 0$ it holds
$$\int_{t_0}^tG_\omega[f](s)\,ds=-{1\over \omega}\int_{t_0}^t f(s)\,ds + {1\over \omega}\left[G_\omega[f](t)-G_{\omega}[f](t_0)\right].$$
Hence, it follows that
  \begin{equation*}
      \begin{split}
          \int_{t_0}^tG[f](s)\,ds 
          =&\,(-1)^{n+1}\prod_{j=1}^{n-1}\gamma^{-1}_j \int_{t_0}^tf(s)\,ds  + \sum_{j=1}^{n-1}{1\over \Gamma_j\gamma_j}\left[G_{\gamma_j}[f](t)-G_{\gamma_j}[f](t_0)\right],
      \end{split}
  \end{equation*}
  in view of $\displaystyle\sum_{j=1}^{n-1}{1\over\Gamma_j\gamma_j}=(-1)^n \prod_{j=1}^{n-1}\gamma^{-1}_j$.
  
We have shown that $\mathcal{ F}(t,z,z',\dots,z^{(n-1)})$  is a polynomial of degree $n$ and all its monomials have degree bigger than 1. Moreover, we wrote explicitly $\mathcal{ F}$ in Lemma \ref{nonlinearp}. Using this, it is easy to see that we can rewrite, for $r(t)=(r_0(t),\dots,r_n(t))$, the non linear part by
\begin{equation}\label{fsdef}
		\mathcal{ F}(t,z,z',\dots,z^{(n-2)}) = 	\mathcal{ F}_a(z,z',\dots,z^{(n-2)}) +	\mathcal{ F}_r(t,z,z',\dots,z^{(n-2)}) \,,
\end{equation}
where
$$
\mathcal{ F}_a(z,z',\dots,z^{(n-2)})= \sum_{i=2}^{n}\sum_{j=0}^{n-i}{i+j \choose j}a_{i+j}\lambda^j\big(B_i(z,z',\dots,z^{(i-1)})-z^{(i-1)}(t)\big)\,,
$$
and
$$
\mathcal{ F}_r(t,z,z',\dots,z^{(n-2)})= \sum_{i=2}^{n-1}\sum_{j=0}^{n-i}{i+j \choose j}r_{i+j}(t)\lambda^j\big(B_i(z,z,\dots,z^{(i-1)})-z^{(i-1)}(t)\big)\,
$$
in view of $r_n\equiv 0$.

Taking into account Lemma \ref{dpol} and Corollary \ref{difertroots}, let $\gamma_j=\lambda_j-\lambda$, $j=1,\dots,n-1$ be the different roots of the characteristic polynomial $P_{\mathcal{D}}$ associated to the differential operator $\mathcal{ D}$ given by \eqref{linearz} and denote by $\alpha_j=\Re\gamma_j\not=0$\, and $\tilde{\alpha}_j=\sum_{i=0}^{n-2}|\alpha_j|^{i} $\,, for $j=1,\dots,n-1$. Note that in this situation $\Gamma_j=\prod_{k=1,k\ne i}^{n-1} (\lambda_i-\lambda_k)$ in view of \eqref{Gammaj}. Recall that $\lambda$ is a fixed root of $P(a;x)$ and  $\tilde{\gamma}=\sum_{k=1}^{n-1}\frac{\tilde{\alpha}_k}{\left|\Gamma_k\right|}$. 
%
\subsection{General results}

\begin{Theorem}\label{orederzun}
	Assume that perturbations $r_k$, $k=0,1,\dots,n-1$ satisfy
	\begin{equation}\label{gpr}
	    \sum_{i=0}^{n-2}\big|G[P(r;\lambda)]^{(i)}(t)\big|\to 0\quad\text{as $t\to +\infty\ \ $ and}
	\end{equation} 
	\begin{equation}\label{cl0}
	\sum_{j=1}^{n-1}\sum_{k=1}^{n-1} \frac{\tilde{\alpha}_j}{k! \left|\Gamma_j\right|}  \bigg\| I_{\alpha_j}\Big[ \frac{\partial^k}{\partial x^k}P(r(\cdot);\lambda) \Big] \bigg\|_\infty <1 \,,
	\end{equation} 
where $P(r;\lambda)$ is given by \eqref{prl}. Then there exists a solution $y$ to equation \eqref{poinca} such that
\begin{equation}\label{dlp}
    \lim_{t\to + \infty} {y_\lambda^{(i)}(t)\over y_\lambda(t)} = \lambda^{i},\qquad i=1,2,\dots,n-1.
\end{equation}
Furthermore, as $t\to + \infty$
\begin{equation}\label{fath1}
\begin{split}
    y_\lambda(t)=\bigg[1&+O\bigg(\sum_{j=1}^{n-1}I_{\alpha_j}\left[\mathcal{L}(\cdot,z)+\mathcal{F}(\cdot,Z)\right] \bigg)\bigg] e^{\lambda(t-t_0)}\exp\left(- \sum_{j=1}^{n-1}{1\over \Gamma_j\gamma_j} G_{\gamma_j}\big[P(r;\lambda)\big](t)\right)\\
    &\times \exp\left( (-1)^n\prod_{k=1}^{n-1}(\lambda_k-\lambda)^{-1}\int_{t_0}^t [P(r;\lambda)(s)+\mathcal{L}(s,z(s))+\mathcal{F}(s,Z(s))]\,ds\right)\,,
\end{split}
\end{equation}
where $\mathcal{L}(\cdot,z)$ and $\mathcal{F}(\cdot,Z)$ are given by \eqref{onedee} and \eqref{formf}, respectively, $z$ satisfies \eqref{dif}, the integral equation
\begin{equation}\label{eiz}
    z=-G\left[P(r;\lambda)+\mathcal{L}(\cdot,z)+\mathcal{F}(\cdot,Z)\right]
\end{equation}
and $z^{(i)}(t)\to0$ as $t\to+\infty$ for all $i=0,\dots,n-2$. In addition, if there exists a constant $0<\beta< \min\{|\alpha_j| \mid 0\leq j\leq n-1\}$  such that
	\begin{equation}\label{cl}
	\sum_{j=1}^{n-1}\sum_{k=1}^{n-1} \frac{\tilde{\alpha}_j}{\left|\Gamma_j\right| k!}  \left\|I_{\alpha_j- \rm sgn(\alpha_j)\beta}\Big[ \frac{\partial^k}{\partial x^k}P(r(\cdot);\lambda) \Big] \right\|_\infty <{1\over 2} \,.
	\end{equation} then $z$ satisfies \eqref{estzi} for $i=0,\dots,n-2$.
\end{Theorem}
\begin{proof}
We shall apply Proposition \ref{fixord} to equation \eqref{dif} by taking $b(t)=-P(r(t);\lambda)$ and $R(t,z,z',\dots,z^{(n-2)}) = - \mathcal{ L}(t,z)-\mathcal{ F}(t,z,z',\dots,z^{(n-2)}) $\,. 
In order to verify \eqref{intgcond}, we need to estimate  $\mathcal{ F}$. Denoting $X=(x_1,\dots,x_i)\in\R^i$ and $Y=(y_1,\dots,y_i)\in\R^i$\,, we have that 
$$
\big\|f_i(X)-f_i(Y)\big\| \leq \sup_{\|c\|\leq\delta} \big\{\| \nabla f_i(c_1,\dots,c_i)\| \big\} \|X-Y\| \,,
$$
where $f_i$ is given by \eqref{bimzi}, $c=(c_1,\dots,c_i)$ and $X,Y\in \overline{B(0;\delta)}$\,, the closed ball of radius $\delta>0$ with centre in $0$\, in $\R^i$. Since the function  $\nabla f_i$ is continuous, it follows that the map
$$m_i:
\delta \mapsto \sup_{\|c\|\leq\delta} \big\{\| \nabla f_i(c_1,\dots,c_i)\| \big\}
$$
is also continuous. Defining $m(\delta):=\max_{i=1,\dots,n-1}\{m_i(\delta)\}\,$, we have an increasing continuous map  $m:[0,\infty) \to [0,\infty)$ such that $m(0)=0$, $m(\delta)>0$ for all $\delta>0$\,, $m(\delta)\to+\infty$ as $\delta\to+\infty$ and 
\begin{equation}\label{mofm}
\|f_i(X)-f_i(Y)\|\leq m(\delta) \| X - Y \| \,,\quad \text{ for all } \, X,\,Y\in \overline{B(0;\delta)}
\end{equation}
for all $i=1,\dots,n-1$. 
Then,  for $\delta>0$ and $z_1,\,z_2\in \mathcal{C}_0^{n-2}$ such that $\|z_i\|_{\mathcal{C}_0^{n-2}}<\delta$\,, for $i=1,2$\,, we have that 
\begin{equation*}
\begin{split}
\Big|\mathcal{ F}_a\big(t,Z_1(t)\big) -\mathcal{ F}_a\big(t,Z_2(t)\big) \Big| \leq & \!\sum_{i=2}^{n}\sum_{j=0}^{n-i}{i+j \choose j}|a_{i+j}|\,|\lambda|^j\big| f_{i-1}(z_1(t),\dots,z_1^{(i-2)}(t)  \\
& \hspace{3.5cm}-f_{i-1}(z_2(t),\dots,z_2^{(i-2)}(t))\big|\, \\
\leq & \sum_{i=2}^{n}\sum_{j=0}^{n-i}{i+j \choose j}|a_{i+j}|\,|\lambda|^j m(\delta) \|Z_1(t)-Z_2(t)\| 
\end{split}
\end{equation*}
Analogously,
\begin{equation*}
\Big|\mathcal{ F}_r\big(t,Z_1(t)\big) -\mathcal{ F}_r\big(t,Z_2(t)\big) \Big| \leq \sum_{i=2}^{n}\sum_{j=0}^{n-i}{i+j \choose j} |r_{i+j}(t)|\, 
|\lambda|^j m(\delta) \|Z_1(t)-Z_2(t)\| \,.
\end{equation*}
Then, it follows that 
\begin{equation}\label{flipzt}
\begin{split}
\Big|\mathcal{ F}\big(t,Z_1(t)\big) -\mathcal{ F}\big(t,Z_2(t)\big) \Big|  &\leq m(\delta) \|Z_1(t)-Z_2(t)\|  \sum_{i=2}^{n}\sum_{j=0}^{n-i}{i+j \choose j}
\big(|a_{i+j}|+|r_{i+j}(t)|\big)|\lambda|^j .
\end{split}
\end{equation}

On the other hand, from \eqref{onedee} one can see that
\begin{equation*}
\begin{split}
\big|\mathcal{ L}(t,z_1)-\mathcal{ L}(t,z_2) \big| 
&\leq \sum_{k=1}^{n-1} \frac{1}{k!}\left|\frac{\partial^{k}}{\partial x^k}P(r(t);\lambda)\right| \|Z_1(t)-Z_2(t)\|\,,
\end{split}
\end{equation*}
with $Z_i=(z_i,z_i',\dots, z_i^{(n-2)})$, $i=1,2$. Thus, considering  
\begin{equation}\label{xisubm}
\begin{split}
\xi_{\delta}(t) &= \sum_{k=1}^{n-1} \frac{1}{k!} \left|\frac{\partial^k}{\partial x^k}P(r(t);\lambda)\right|  +  m(\delta)\sum_{i=2}^{n}\sum_{j=0}^{n-i}{i+j \choose j}  \big(|a_{i+j}|+|r_{i+j}(t)|\big)|\lambda|^j \,,
\\
&= \sum_{k=1}^{n-1} \frac{1}{k!} \left|\frac{\partial^k}{\partial x^k}P(r(\cdot);\lambda)\right|  +  m(\delta)\sum_{i=2}^{n}\sum_{k=0}^{i-2}{i\choose k}  \big(|a_{i}|+|r_{i}(t)|\big)|\lambda|^k,
\end{split}
\end{equation}
one obtains that 
$$
|\mathcal{ L}(t,z_1)+\mathcal{ F}(t,Z_1(t))-\mathcal{ L}(t,z_2)-\mathcal{ F}(t,Z_2(t)) | \leq \xi_{\delta}(t)\|Z_1(t)-Z_2(t)\|\,,
$$
for all $\delta>0$ and  $Z_1(t)\,,Z_2(t)\in \overline{B(0;\delta)}$\,. Therefore, since $R=-\mathcal{L}-\mathcal{F} $, one has that $R$ satisfy the Lipschitz condition  \eqref{lipchitz}.
On the other hand, for any $0\le \beta <\min\{|\alpha_j| \ : \ 0\le j\le n-1\}$ we have that 
\begin{equation*}
\begin{split}
I_{\alpha_j-{\rm sgn}(\alpha_j)\beta}\big[ \xi_{\delta}\big](t) &\leq m(\delta)\sum_{i=2}^{n}\sum_{k=0}^{i-2}{i\choose k}  |\lambda|^{k}\Big({|a_{i}|\over |\alpha_j-{\rm sgn}(\alpha_j)\beta|} 
+ I_{\alpha_j-{\rm sgn}(\alpha_j)\beta}\big[|r_{i}|\big] (t) \Big) 
\\
& \hspace{2.1cm} + \sum_{k=1}^{n-1} \frac{1}{k!} I_{\alpha_j-{\rm sgn}(\alpha_j)\beta}\Big[ \frac{\partial^k}{\partial x^k}P(r(\cdot);\lambda) \Big]\!(t) \,,
\end{split}
\end{equation*}
in view of $\displaystyle I_{\alpha_j-{\rm sgn}(\alpha_j)\beta}\big[1\big](t)\le {1\over |\alpha_j-{\rm sgn}(\alpha_j)\beta|}\quad\text{for all $t\ge t_0$}$. Let us stress that latter inequality also includes the case $\beta=0$. Now, let us define
$$L_\beta:=\sum_{j=1}^{n-1}\sum_{k=1}^{n-1} \frac{\tilde{\alpha}_j}{\left|\Gamma_j\right| k!}  \left\| I_{\alpha_j- \rm sgn(\alpha_j)\beta}\Big[ \frac{\partial^k}{\partial x^k}P(r(\cdot);\lambda) \Big] \right\|_\infty$$
and
$$Q_\beta := \sum_{j=1}^{n-1}  \sum_{i=2}^{n}\sum_{k=0}^{i-2} \frac{\tilde{\alpha}_j}{\left|\Gamma_j\right|} {i\choose k}  |\lambda|^{k}\Big({|a_{i}|\over |\alpha_j-{\rm sgn}(\alpha_j)\beta|} 
+ \left\|I_{\alpha_j-{\rm sgn}(\alpha_j)\beta}\big[|r_{i}|\big]\right\|_\infty \Big)$$
so that, \eqref{cl0} is equivalent to $L_0<1$ and \eqref{cl} is equivalent to $L_\beta<\dfrac{1}{2}$. Note that \eqref{cl0} implies that $I_\gamma[r_i]$ is bounded for all $i=2,\dots,n-1$ and any $\gamma>0$, so that $Q_\beta$ is well defined.
Thus, we have that for all $\delta>0$ and any $0\le \beta <\min\{|\alpha_j| \, : \, 0\le j\le n-1\}$
\begin{equation*}
\begin{split}
\sum_{j=0}^{n-1}\left|\frac{\tilde{\alpha}_j}{\Gamma_j}\right| I_{\alpha_j-{\rm sgn}(\alpha_j)\beta}\big[\xi_{\delta}\big](t)  \leq&  
 m(\delta)Q_\beta +L_\beta\,,
\end{split}
\end{equation*}
Note that for any function $f$ it holds $I_{\alpha_j}[f](t)\le I_{\alpha_j-\sgn(\alpha_j)\beta}[f](t),$ so that, $L_0\le L_\beta$ and $Q_0\le Q_\beta$.
Hence, by using \eqref{cl0}, $\beta=0$ and properties of the map $m$\, there exists a $\delta_0>0$ such that
\begin{equation*}
0<m(\delta)<{1\over Q_0}(1-L_0)\quad\text{ for all $0<\delta<\delta_0$ y } \quad   m(\delta_0)={1\over Q_0}(1-L_0).
\end{equation*}
Therefore, taking any $0<M<\delta_0$ we have that $m(M)<\frac{1}{Q_0}(1-L_0)$\, and we can choose $\epsilon_0=m(M)Q_0+L_0<1$. Now, by using \eqref{gpr} we can choose $t_0$ so that 
conditions \eqref{intgcond} in Proposition \ref{fixord} are satisfied. Therefore, there exists a solution $z$ of equation \eqref{poincaz} (which is the same to \eqref{dif}) such that $z^{(i)}(t)\to 0 $ as $t\to+\infty$ for all $i=0,1,\dots,n-2$. Also, $z$ satisfies the integral equation \eqref{eiz}. Hence, $y_\lambda$ given by \eqref{yeilz} is a solution to \eqref{poinca}. Moreover, we get that from \eqref{ynb}  and \eqref{bilz}
\begin{equation*}
    \begin{split}
     \frac{y_\lambda^{(i)}(t)}{y_\lambda(t)} 
     &=\lambda^{i} + \lambda^{i-1} z(t) + \sum_{j=2}^{i}{i \choose j} \lambda^{i-j}B_{j}\big(z(t),z'(t),\dots,z^{(j-1)}(t)\big)\,.   
    \end{split}
\end{equation*}
for $i=1,2,\dots,n-1$. The asymptotic behavior \eqref{dlp} follows from $z^{(i)}(t)\to 0 $ as $t\to+\infty$ for all $i=0,1,\dots,n-2$ and in view of $B_j(0,\dots,0)=0$ for all $j\ge 1$. Integrating \eqref{eiz} we find that
  \begin{equation*}
      \begin{split}
          \int_{t_0}^t z(s)\,ds
          =&\,(-1)^n\prod_{j=1}^{n-1}\gamma^{-1}_j \int_{t_0}^t\big[P(r;\lambda)(s)+\mathcal{L}(s,z(s))+\mathcal{F}(s,Z(s))\big]\,ds \\
          &\,- \sum_{j=1}^{n-1}{1\over \Gamma_j\gamma_j}\left\{G_{\gamma_j}\big[P(r;\lambda)+\mathcal{L}(\cdot,z)+\mathcal{F}(\cdot,Z)\big](t)\right.\\
          &\left.\qquad\qquad- G_{\gamma_j}\big[P(r;\lambda)+\mathcal{L}(\cdot,z)+\mathcal{F}(\cdot,Z)\big](t_0)\right\}
      \end{split}
  \end{equation*}
Thus, we conclude \eqref{fath1}.

On the other hand, 
by using \eqref{cl} and again by properties of the map $m$\, there exists a $0<\delta_1\le\delta_0$ such that
\begin{equation*}
0<m(\delta)<{1\over Q_\beta}\Big({1\over 2}-L_\beta\Big)\quad\text{ for all $0<\delta<\delta_1$ y } \quad   m(\delta_1)={1\over Q_\beta}\Big({1\over 2}-L_\beta\Big).
\end{equation*}
Therefore, taking $M$ smaller than the previous one if necessary, so that $0<M<\delta_1$ we have that $m(M)<\frac{1}{Q_\beta}({1\over 2}-L_\beta)$ and 
condition \eqref{intgcond} in Proposition \ref{fixord} is satisfied. Therefore, the solution $z$ of equation  \eqref{poincaz} (which is the same to \eqref{dif}) also satisfies
$z^{(i)}=O\big( (I_{\beta}+I_{-\beta})\big[P(r;\lambda)] \big)$\, for $i=0,\dots,n-2$\,. This finishes the proof.
\end{proof}
%

Let us stress that $\lambda$ is a fixed root of $P(a;x)$ and hypothesis \eqref{gpr}, \eqref{cl0} and \eqref{cl} are satisfied with respect to $\lambda$. Thus, we have recovered Poincar\'e's result \cite{Po}. Moreover, we have provided an asymptotic formula \eqref{fath1} for the found solution $y(t)$. Let us stress that assumption \eqref{cl0} implies the existence of $z$ and \eqref{cl} gives us an estimate for $z$. Moreover, \eqref{cl} implies \eqref{cl0}. A simpler asymptotic formula is deduced from \eqref{fath1}, which is \eqref{afintro} in view of
\begin{equation*}
      \begin{split}
          \int_{t_0}^t z(s)\,ds
          =&\,(-1)^n\prod_{j=1}^{n-1}\gamma^{-1}_j \int_{t_0}^t\big[P(r;\lambda)(s)+\mathcal{L}(s,z(s))+\mathcal{F}(s,Z(s))\big]\,ds + c + o(1),
      \end{split}
  \end{equation*}
as $t\to+\infty$, in view of \eqref{gpr} and  $G_{\gamma_j}[\mathcal{L}(\cdot,z)](t),G_{\gamma_j}[\mathcal{F}(\cdot,Z)](t)\to 0$ for any $j=1,\dots,n-1$ as $t\to+\infty$.

\medskip
Now, we shall prove the existence of fundamental system of solutions $\{y_i\}_i$, $i=1,\dots, n$, assuming conditions of Theorem \ref{orederzun} for each $\lambda_i$, $i=1,\dots, n$, the roots of $P(a;x)$. For simplicity, we will only consider assumption \eqref{cl}. In this way, we recover Perron's result \cite{Perr}. To this purpose, we shall denote $P(r;\lambda_i)$, $\mathcal{D}_i$, $P_{\mathcal{D}_i}$ and $\mathcal{F}_i$ by \eqref{prl}, \eqref{linearz}, \eqref{defdpol} and \eqref{deff} with $\lambda=\lambda_i$ respectively.  We shall assume that $\Re(\lambda_i)\ne \Re(\lambda_j)$ for all $i\ne j$. Also, consider $\mathcal{N}(i):=\{1,\dots,n\}\setminus \{i\}$ and the bijection $\sigma_i:\mathcal{N}(i)\to \{1,\dots,n-1\}$ defined by
$$\sigma_i(j)=\begin{cases}
j&\text{if }j<i\\
j-1&\text{if }i<j
\end{cases}\qquad\text{and}\qquad \sigma_i^{-1}(j)=\begin{cases}
j&\text{if }j<i\\
j+1&\text{if }i<j
\end{cases}.$$
Thus, fixing $i\in\{1,\dots,n\}$ we have that $\gamma_{ij}=\lambda_{\sigma_i^{-1}(j)}-\lambda_i$, $j=1,\dots,n-1$ are the roots of polynomial $P_{\mathcal{D}_i}$ as shown in Lemma \ref{dpol}. Furthermore, Green's operator for equation $\mathcal{D}_iz=f$ is given by
$$
G_i\big[f\big](t):=\int_{t_0}^{\infty} \mathcal{G}_i(s,t)f(s)\,ds = \sum_{j=1}^{n-1}\frac{1}{\Gamma_{ij}}G_{\gamma_{ij}}\big[f\big](t)\,,
$$
where
\begin{equation*}
\Gamma_{ij}:= \prod_{k= 1,\,k\not=j}^{n-1} (\gamma_{ij}-\gamma_{ik})\quad\text{and}\quad G_i(t,s)= \sum_{j=1}^{n-1}\frac{1}{\Gamma_{ij} }g_{\gamma_{ij}}(t,s),
\end{equation*}
with $g_\gamma$ is given by \eqref{defg}. Denote $\alpha_{ij}=\Re(\gamma_{ij})$ and  $\displaystyle\tilde\alpha_{ij}=\sum_{k=1}^{n-2}|\gamma_{ij}^k|$. 
	
\begin{Theorem}\label{base}	
Assume that perturbations $r_k$, $k=0,1,\dots,n-1$ satisfy for every $i=1,\dots,n$
	\begin{equation}\label{gpri}
	    \sum_{j=0}^{n-2}\big|G_i[P(r;\lambda_i)]^{(j)}(t)\big|\to 0\quad\text{as $t\to +\infty\ \ $ and}
	\end{equation} 
there exist constants $0<\beta_i< \min\{|\alpha_{ij}| \mid 0\leq j\leq n-1\}$ for $i=1,\dots,n$ such that
	\begin{equation}\label{cli}
	\sum_{j=1}^{n-1} \sum_{k=1}^{n-1} \frac{\tilde{\alpha}_{ij} }{\left|\Gamma_{ij} \right| k! } \left\| I_{\alpha_{ij}- {\rm sgn}(\alpha_{ij})\beta_i}\Big[ \frac{\partial^k}{\partial x^k}P(r(\cdot);\lambda_i) \Big]\! \right\|_\infty <{1\over 2},
	\end{equation} 
where $P(r;\lambda_i)$ is given by \eqref{prl} with $\lambda=\lambda_i$. Then there exists a fundamental system of solutions $y_i$, $i=1,\dots,n$ to equation \eqref{poinca} such that
\begin{equation}\label{dlpi}
    \lim_{t\to + \infty} {y_i^{(j)}(t)\over y_i(t)} = \lambda_i^{j},\qquad j=1,2,\dots,n-1.
\end{equation}
Furthermore, as $t\to + \infty$ for every $i=1,\dots,n$ \begin{equation}\label{fath2}
    \begin{split}
        y_i(t)=&\,(1+o(1)) e^{\lambda_i(t-t_0)}\\
        &\times \exp\left((-1)^n\prod_{j\in\mathcal{N}(i)}(\lambda_j-\lambda_i)^{-1}\int_{t_0}^t [P(r;\lambda_i)(s)+\mathcal{L}_i(s,z_i(s))+\mathcal{F}_i(s,Z_i(s))]\,ds\right)\,,
    \end{split}
\end{equation} 
where $\mathcal{L}_i(\cdot,z_i)$ and $\mathcal{F}_i(\cdot,Z_i)$ are given by \eqref{onedee} and \eqref{formf} with $\lambda=\lambda_i$, respectively, $z_i$ satisfies differential equation
$$\mathcal{D}_iz+P(r;\lambda_i)+\mathcal{L}_i(t,z)+\mathcal{F}_i\big(t,z,\dots,z^{(n-2)}\big)=0,$$
the integral equation
$$z=- G_i[P(r;\lambda_i) + \mathcal{L}_i(\cdot,z)+\mathcal{F}_i(\cdot,Z)],$$
$z_i^{(j)}(t)\to0$ as $t\to+\infty$ for all $j=1,\dots,n-1$ and $z_i^{(j)}=O\big((I_{\beta_i}+I_{-\beta_i})[P(r;\lambda_i)]\big)$ so that
	$$\frac{y_i^{(j)}(t)}{y_i(t)}=\lambda_i^{j} + O\big((I_{\beta_i}+I_{-\beta_i})[P(r;\lambda_i)]\big)$$ 
	for $j=0,1,\dots,n-2$\,.
	\end{Theorem}
\begin{proof}
From assumptions \eqref{gpri}, \eqref{cli} and Theorem \ref{theo1}, it follows the existence of $y_i$, $i=1,\dots,n$ satisfying \eqref{dlpi}. Hence, we shall prove that $\{y_i\}_i$ is a fundamental system of solutions. Denote by $W$ the Wronskian of the functions $y_i$'s. Note that \eqref{dlpi} is equivalent to
	$$
	\frac{y_i^{(j)}(t)}{y_i(t)} = \lambda_i^{j}+o(1)
	$$
	as $t\to+\infty$. Thus, it follows that as $t\to+\infty$
	\begin{equation*}
		\begin{split}
		W(t) &= y_1(t)y_2(t)\cdots y_n(t) \left[
		\det
		\left(\begin{matrix}  1 & 1&\cdots & 1 \\
		\lambda_1&\lambda_2&\cdots&\lambda_n \\
		\vdots & \vdots & \cdots &\vdots \\
		\lambda_1^{n-1}& \lambda_2^{n-1} & \cdots & \lambda_{n}^{n-1}
		\end{matrix}\right) +o(1)\right]\\
		&=  y_1(t)y_2(t)\cdots y_n (t) \left[ \prod_{i>j}^{n} (\lambda_i-\lambda_j) +o(1)\right]
		\end{split}
	\end{equation*}
	where we have used the Vandermonde rule for the determinant. From assumptions on $\lambda_i$'s it follows that $W(t)\not=0$\,. This completes the proof.	
\end{proof}

In view of Lemma \ref{oer}, equation \eqref{poinca} admits perturbations satisfying either $r_i(t)\to 0$ as $t\to+\infty$ or $r_i\in L^p[t_0,+\infty)$ for some $p\ge 1$.

\subsection{Levinson's and Hartman-Wintner's type results}

In this subsection we present some consequences of our general result, Theorem \ref{orederzun}. First, a weaker version of Levinson's and Hartman-Wintner's type results is stated.

\begin{Theorem}\label{pplev}
Assume that $r_i$, $i=0,\dots,n-1$ in equation \eqref{poinca} satisfy $P(r;\lambda)\in L^1[t_0,\infty)$ and \eqref{cl}. Then there exists a solution $y$ to equation \eqref{poinca} satisfying \eqref{dlp}. Furthermore, 
$$y_\lambda(t)=\bigg[1+O\left(\int_{t}^{+\infty}(I_{\beta}+I_{-\beta})[P(r;\lambda)](s)\, ds \right)\bigg] e^{\lambda(t-t_0)}$$
and its derivatives satisfy \eqref{yjioyi}.
\end{Theorem}

\begin{proof} Thanks to Lemma \ref{oer} we have that  $G_{\gamma_j}[P(r;\lambda)](t)\to0$ as $t\to+\infty$ and \linebreak $G_{\gamma_j}[P(r;\lambda)]\in L^1[t_0,+\infty)$, for all $j=1,\dots,n-1$. Hence, assumption \eqref{gpr} is fulfilled. In view of condition \eqref{cl}, there exists a solution $y_\lambda$ to equation \eqref{poinca} in the form \eqref{yeilz} with $z^{(i)}=O\big((I_{\beta}+I_{-\beta})[P(r;\lambda)]\big)$ for all $i=0,\dots,n-2$. Then, we have that $z\in L^1[t_0,+\infty)$, since $(I_{\beta}+I_{-\beta})[P(r;\lambda)]\in L^1[t_0,\infty)$. 
Note that $|e^x-1|\leq C|x|$ for $x \in\mathbb{R}$ on compact subsets, so that,
\begin{equation*}
    \bigg|\exp\left(-\int_t^{+\infty} z(s)\, ds\right)-1\bigg|\le C \left|\int_t^{+\infty} z(s)\, ds\right|\leq C\int_{t}^{+\infty}(I_{\beta}+I_{-\beta})[P(r;\lambda)](s)\, ds.
\end{equation*}
The result follows from
$$y_\lambda(t)=\left[1+\varepsilon(t)\right]e^{\lambda(t-t_0)+\int_{t_0}^{+\infty}z(s)\, ds}\qquad\text{with}\qquad \varepsilon(t)=\exp\left(-\int_t^{+\infty} z(s)\, ds\right)-1.$$
\end{proof}

\begin{Theorem}\label{pphw}
Assume that $r_i$, $i=0,\dots,n-1$ in equation \eqref{poinca} satisfy $P(r;\lambda)\in L^p[t_0,\infty)$ and $\dfrac{\partial^k}{\partial x^k}[P(r;\lambda)]\in L^{p_k}[t_0,+\infty)$ with $p,p_k\in(1,2]$ for all $k=1,\dots,n-1$. Then there exists a solution $y_\lambda$ for the equation \eqref{poinca} satisfying \eqref{dlp}. Furthermore, as $t\to + \infty$ 
\begin{equation*}
\begin{split}
y_\lambda(t)=&\,(1+\epsilon(t))e^{\lambda(t-t_0)}\\
&\, \times\exp\left(- \sum_{j=1}^{n-1}{1\over \Gamma_j\gamma_j} G_{\gamma_j}\big[P(r;\lambda)\big](t) +(-1)^n\prod_{k=1}^{n-1}(\lambda_k-\lambda)^{-1}\int_{t_0}^t P(r;\lambda)(s)\,ds\right),
\end{split}
\end{equation*}
where
$\displaystyle \epsilon(t)=O\bigg(\sum_{j=1}^{n-1} I_{\alpha_j}[\mathcal{L}(\cdot,z)+ \mathcal{F}(\cdot ,Z)] +\int_t^{+\infty}\left[\big|\mathcal{L}(s,z(s))\big|+ \big|\mathcal{F}(s,Z(s))\big|\right]ds\bigg).$
\end{Theorem}

\begin{proof}
Again thanks to Lemma \ref{oer} we have that  $G_{\gamma_j}[P(r;\lambda)](t)\to0$ as $t\to+\infty$ and \linebreak $G_{\gamma_j}[P(r;\lambda)]\in L^p[t_0,+\infty)$, for all $j=1,\dots,n-1$. Hence, assumption \eqref{gpr} is fulfilled. In view of $\dfrac{\partial^k}{\partial x^k}[P(r;\lambda)]\in L^{p_k}[t_0,+\infty)$ with $p,p_k\in(1,2]$ for all $k=1,\dots,n-1$, condition \eqref{cl} is also fulfilled, so that there exists a solution $y$ to equation \eqref{poinca} satisfying \eqref{fath1} with $z^{(i)}=O\big((I_{\beta}+I_{-\beta})[P(r;\lambda)]\big)$ for all $i=0,\dots,n-2$. Then, we have that $z\in L^q[t_0,+\infty)$, since $(I_{\beta}+I_{-\beta})[P(r;\lambda)]\in L^q[t_0,\infty)$ for all $q\geq p$. Hence, for each $p_k$ we find $q_k\geq 2$ such that $\dfrac{1}{p_k}+\dfrac{1}{q_k}=1$ and $z^{(k-1)}\in L^{q_k}[t_0,+\infty)$, so that $\dfrac{\partial^k}{\partial x^k}[P(r;\lambda)] z
^{(k-1)}\in L^1[t_0,+\infty)$ for all $k=1,\dots,n$. Thus, it follows that $\mathcal{L}(\cdot,z)\in L^1[t_0,+\infty)$. Similarly, $\mathcal{F}(\cdot,Z)\in L^1[t_0,+\infty)$ in view of Lemma \ref{nonlinearp}. Then
  \begin{equation*}
      \begin{split}
          \int_{t_0}^t z(s)\,ds
          =&\,c - \sum_{j=1}^{n-1}{1\over \Gamma_j\gamma_j} G_{\gamma_j}\big[P(r;\lambda)+\mathcal{L}(\cdot,z)+\mathcal{F}(\cdot,Z)\big](t)+(-1)^n\prod_{j=1}^{n-1}\gamma^{-1}_j \int_{t_0}^t P(r;\lambda)(s) \,ds\\
          &\,+(-1)^n \prod_{j=1}^{n-1}\gamma^{-1}_j \int_{t}^{+\infty} \big[\mathcal{L}(s,z(s))+\mathcal{F}(s,Z(s))\big]\,ds,
      \end{split}
  \end{equation*}
for some constant $c$ and the results follows by using that $|e^x-1|\leq C|x|$ for $x \in\mathbb{R}$ on compact subsets.
\end{proof}

Note that an explicit estimate of $\epsilon$ in terms of perturbations $r_i$'s (independent of $z$) can be found by using $z^{(i)}=O\big((I_{\beta}+I_{-\beta})[P(r;\lambda)]\big)$ and expressions of $\mathcal{L}(\cdot,z)$ and $\mathcal{F}(\cdot,Z)$.

\medskip

Following ideas presented in \cite{FP2}, we shall take an  advantage of integral equation for $z$. Precisely, denote
\begin{equation}
    \theta=-G[P(r;\lambda)]\qquad\text{and}\qquad R(\cdot,Z)=\mathcal{L}(\cdot,z) + \mathcal{F}(\cdot,Z),
\end{equation}
so that, $z$ is a solution to \eqref{dif} if and only if $u=z-\theta$ is a solution to
\begin{equation}\label{eu}
    u=-G\left[R(\cdot,\theta)+\mathcal{L}(\cdot,u)+\tilde{\mathcal{F}}(\cdot,U)\right],
\end{equation}
where $\tilde{\mathcal{F}}(\cdot,U)=\mathcal{F}(\cdot,U+\Theta)- \mathcal{F}(\cdot,\Theta)$, $U=(u,u',\dots,u^{(n-2)})$ and $\Theta=(\theta,\theta',\dots,\theta^{(n-2)})$. Applying Proposition \ref{fixord} to \eqref{eu}, we obtain the following fact.

\begin{Theorem}\label{tut}
Assume that $r_i$, $i=0,\dots,n-1$ in equation \eqref{poinca} satisfy $R(\cdot,\theta)\in L^1[t_0,\infty)$ and \eqref{cl}. Then there exists a solution $y_\lambda$ for the equation \eqref{poinca} satisfying \eqref{dlp} and as $t\to + \infty$ 
\begin{equation*}
    \begin{split}
        y_\lambda(t)=&\, [1+\varepsilon(t)] e^{\lambda(t-t_0)}\\
        &\,\times \exp\bigg(- \sum_{j=1}^{n-1}{1\over \Gamma_j\gamma_j} G_{\gamma_j}\big[P(r;\lambda)\big](t)  +(-1)^n\prod_{k=1}^{n-1}(\lambda_k-\lambda)^{-1}\int_{t_0}^t P(r;\lambda)(s)\,ds\bigg),
    \end{split}
\end{equation*}
with $\displaystyle\varepsilon(t)=O\left(\int_{t}^{+\infty}(I_{\beta}+I_{-\beta})[R(\cdot,\theta)](s)\, ds \right).$
\end{Theorem}

\begin{proof} Thanks to Lemma \ref{oer} we have that  $G_{\gamma_j}[R(\cdot,\theta)](t)\to0$ as $t\to+\infty$ and \linebreak $G_{\gamma_j}[R(\cdot,\theta)]\in L^1[t_0,+\infty)$, for all $j=1,\dots,n-1$. Hence, assumption $G[R(\cdot,\theta)](t)\to0$ as $t\to+\infty$ is fulfilled. In view of condition \eqref{cl} and Proposition \ref{fixord}, there exists a solution $u$ to equation \eqref{eu} with $u^{(i)}=O\big((I_{\beta}+I_{-\beta})[R(\cdot,\theta)]\big)$ for all $i=0,\dots,n-2$. Then, we have that $u\in L^1[t_0,+\infty)$, since $(I_{\beta}+I_{-\beta})[R(\cdot,\theta)]\in L^1[t_0,\infty)$. Using \eqref{yeilz} and $z=u+\theta$ the result follows from
$$y_\lambda(t)=\left[1+\varepsilon(t)\right]e^{\lambda(t-t_0)+\int_{t_0}^t \theta(s)\, ds+\int_{t_0}^{+\infty}u(s)\, ds}\quad\text{with}\quad \varepsilon(t)=\exp\left(-\int_t^{+\infty} u(s)\, ds\right)-1$$
and
  \begin{equation*}
      \begin{split}
          \int_{t_0}^t \theta(s)\,ds
          =&\,- \sum_{j=1}^{n-1}{1\over \Gamma_j\gamma_j} G_{\gamma_j}\big[P(r;\lambda)\big](t) -\prod_{j=1}^{n-1}\gamma^{-1}_j \int_{t_0}^t P(r;\lambda)(s) \,ds + c,
      \end{split}
  \end{equation*}
for some constant $c$.
\end{proof}

\begin{Theorem}\label{teots}
Assume that $r_i$, $i=0,\dots,n-1$ in equation \eqref{poinca} satisfy $R(\cdot,\theta)\in L^p[t_0,\infty)$ and $\dfrac{\partial^k}{\partial x^k}[P(r;\lambda)]\in L^{p_k}[t_0,+\infty)$ with $p,p_k\in(1,2]$ for all $k=1,\dots,n$ and \eqref{cl}. Then there exists a solution $y_\lambda$ for the equation \eqref{poinca} satisfying \eqref{dlp}. Furthermore, as $t\to + \infty$ 
\begin{equation*}
    \begin{split}
        y_\lambda(t)= (1+\varepsilon(t))e^{\lambda(t-t_0)}&\exp\left(- \sum_{j=1}^{n-1}{1\over \Gamma_j\gamma_j} G_{\gamma_j}\big[P(r;\lambda)+R(\cdot, \theta) \big](t)\right)\\
        & \times \exp\left((-1)^n\prod_{k=1}^{n-1}(\lambda_k-\lambda)^{-1}\int_{t_0}^t\left[ P(r;\lambda)(s)+R(s,\theta(s))\right]\,ds\right),
    \end{split}
\end{equation*}
with $\displaystyle \varepsilon(t)=O\bigg(\sum_{j=1}^{n-1} I_{\alpha_j}[\mathcal{L}(\cdot,u)+ \tilde{\mathcal{F}}(\cdot ,U)] +\int_t^{+\infty}\left[\big|\mathcal{L}(s,u(s))\big|+ \big|\tilde{\mathcal{F}}(s,U(s))\big|\right]ds\bigg).$

\end{Theorem}

\begin{proof}
Again thanks to Lemma \ref{oer} we have that  $G_{\gamma_j}[R(\cdot,\theta)](t)\to0$ as $t\to+\infty$ and \linebreak $G_{\gamma_j}[R(\cdot,\theta)]\in L^p[t_0,+\infty)$, for all $j=1,\dots,n-1$. $G[R(\cdot,\theta)](t)\to0$ as $t\to+\infty$ is fulfilled. In view of condition \eqref{cl} and Proposition \ref{fixord}, there exists a solution $u$ to equation \eqref{eu} with $u^{(i)}=O\big((I_{\beta}+I_{-\beta})[R(\cdot,\theta)]\big)$ for all $i=0,\dots,n-2$. Then, we have that $u\in L^q[t_0,+\infty)$, since $(I_{\beta}+I_{-\beta})[R(\cdot,\theta)]\in L^q[t_0,\infty)$ for all $q\geq p$. Hence, for each $p_k$ we find $q_k\geq 2$ such that $\dfrac{1}{p_k}+\dfrac{1}{q_k}=1$ and $u^{(k-1)}\in L^{q_k}[t_0,+\infty)$, so that $\dfrac{\partial^k}{\partial x^k}[P(r;\lambda)] u
^{(k-1)}\in L^1[t_0,+\infty)$ for all $k=1,\dots,n$. Thus, it follows that $\mathcal{L}(\cdot,u)\in L^1[t_0,+\infty)$. Similarly, $\tilde{\mathcal{F}}(\cdot,U)\in L^1[t_0,+\infty)$ in view of Lemma \ref{nonlinearp}. Using \eqref{yeilz} and $z=\theta+u=\theta-G[R(\cdot,\theta)+\mathcal{L}(\cdot,u)+\tilde{\mathcal{F}}(\cdot,U)]$, we obtain that
  \begin{equation*}
      \begin{split}
          \int_{t_0}^t u(s)\,ds
          =&\,c - \sum_{j=1}^{n-1}{1\over \Gamma_j\gamma_j} G_{\gamma_j}\big[R(\cdot,\theta)+\mathcal{L}(\cdot,u)+\tilde{\mathcal{F}}(\cdot,U)\big](t)+(-1)^n\prod_{j=1}^{n-1}\gamma^{-1}_j \int_{t_0}^t R(\cdot,\theta)(s) \,ds \\
          &\,+(-1)^n \prod_{j=1}^{n-1}\gamma^{-1}_j \int_{t}^{+\infty} \big[\mathcal{L}(s,u(s))+\tilde{\mathcal{F}}(s,U(s))\big]\,ds
      \end{split}
  \end{equation*}
for some constant $c$ and the result follows.
\end{proof}

\subsection{Harris-Lutz type result}

Finally, we discuss a Harris-Lutz's type result, see \cite{FP1,FP3} for cases $n=2,3$ and \cite{CHP2} for $n=4$. We shall assume that $r_i\in L^p[t_0,+\infty)$, $i=0,1,\dots,n-1$ for any $p>2$. Taking into account previous results, we have that there exists a solution $y$ to equation \eqref{poinca} satisfying \eqref{dlp} with $z^{(i)}=O\big((I_\beta+I_{-\beta})[P(r;\lambda)]\big)$ for all $i=0,\dots,n-2$. Then, $z^{(i)}\in L^q[t_0,+\infty)$ for all $i=0,1,\dots,n-2$ and any $q\ge p$. Now, we will use decomposition \eqref{dfi} in order to decompose $\mathcal{F}(t,Z)$ in terms of integrable functions of decreasing order. Replacing \eqref{dfi} in \eqref{fsdef} we get that
\begin{equation}\label{dfz}
    \begin{split}
        \mathcal{F}(t,Z)
        =&\,\sum_{k=2}^n\sum_{i=k}^n\sum_{j=0}^{n-i}{i+j \choose j} (a_{i+j}+r_{i+j}(t))\lambda^j h_{k,i-1}(z,z',\dots,z^{(i-2)})\\
        =&\, \sum_{i=2}^n \tilde a_i h_{2,i-1}(z,z',\dots,z^{(i-2)}) \\ 
        &\,+ \sum_{k=3}^n\bigg[\sum_{i=k}^n \tilde a_i h_{k,i-1}(z,z',\dots,z^{(i-2)}) + \sum_{i=k-1}^{n-1} h_{k-1,i-1}(z,z',\dots,z^{(i-2)})\tilde r_i(t) \bigg], 
    \end{split}
\end{equation}
in view of $r_n\equiv 0$. For simplicity, we have denoted
$$\tilde{a}_i=\sum_{j=0}^{n-i}{i+j \choose j} a_{i+j}\lambda^j\qquad\text{and}\qquad \tilde{r}_i(t)=\sum_{j=0}^{n-i-1}{i+j \choose j} \lambda^j r_{i+j}(t).$$
Then, either $h_{k,i-1}(z,z',\dots,z^{(i-2)})$ belongs to $L^{p/k}[t_0,+\infty)$ or  $L^{1}[t_0,+\infty)$ depending on whether or not $p>k$. Thus, if $p>2$ then $\displaystyle \sum_{i=2}^n \tilde{a}_i h_{2,i-1}(z,z',\dots,z^{(i-2)}) \in L^{p/2}[t_0,+\infty)$
and if $p>k$ then
$$\sum_{i=k}^n \tilde{a}_i h_{k,i-1}(z,z',\dots,z^{(i-2)})+ \sum_{i=k-1}^n h_{k-1,i-1}(z,z',\dots,z^{(i-2)})\tilde{r}_i\in L^{p/k}[t_0,+\infty),$$
in view of $\tilde r_i\in L^p[t_0,+\infty)$. In order to illustrate main ideas assume that $p\in(2,3]$. Hence, $\mathcal{L}(\cdot,z),\mathcal{F}(\cdot,Z)\in L^{p/2}[t_0,+\infty)$, so, $z=\theta_1 + \psi_1$ where $\theta_1^{(k)}\in L^{p}[t_0,+\infty)$ and $\psi_1^{(k)}\in L^{p/2}[t_0,+\infty)$, $k=0,1,\dots,n-2$, with $\psi_1=-G[\mathcal{L}(\cdot,z)+\mathcal{F}(\cdot,Z)]$. Replacing in the integral equation for $z$ we get that
$$z=\theta_1  -G[\mathcal{L}(\cdot,z)+\mathcal{F}(\cdot,Z)]=\theta_1 -G[\mathcal{L}(\cdot,\theta_1 + \psi_1)+\mathcal{F}(\cdot,\Theta_1 + \Psi_1)].$$
Note that $\mathcal{L}(\cdot,\theta_1 + \psi_1)=\mathcal{L}(\cdot,\theta_1) + \mathcal{L}(\cdot,\psi_1)$, $\mathcal{L}(\cdot,\theta_1)\in L^{p/2}[t_0,+\infty)$, $\mathcal{L}(\cdot,\psi_1)\in L^{1}[t_0,+\infty)$ and
\begin{equation*}
    \begin{split}
        \mathcal{F}(\cdot,\Theta_1 + \Psi_1)=&\ \sum_{i=2}^n \tilde{a}_i h_{2,i-1}(\theta_1,\theta_1',\dots,\theta_1^{(i-2)})\\
        &\ + \mathcal{F}(\cdot,\Theta_1 + \Psi_1) -\sum_{i=2}^n \tilde{a}_i h_{2,i-1}(\theta_1,\theta_1',\dots,\theta_1^{(i-2)}),
    \end{split}
\end{equation*}
$\sum_{i=2}^n \tilde{a}_i h_{2,i-1}(\theta_1,\theta_1',\dots,\theta_1^{(i-2)}) \in L^{p/2}[t_0,+\infty)$ and 
$$\mathcal{F}(\cdot,\Theta_1 + \Psi_1) -\sum_{i=2}^n \tilde{a}_i h_{2,i-1}(\theta_1,\theta_1',\dots,\theta_1^{(i-2)})\in L^{1}[t_0,+\infty).$$
Therefore, it follows that $z=\theta_1+\theta_2+\psi_2$, where $\theta_i\in L^{p/i}[t_0,+\infty)$ for $i=1,2$, they are given by 
$$\theta_1=-G[P(r;\lambda)]\quad\text{and}\quad\theta_2=-G\left[\mathcal{L}(\cdot,\theta_1)+\sum_{i=2}^n \tilde{a}_i h_{2,i-1}(\theta_1,\theta_1',\dots,\theta_1^{(i-2)})\right]$$
and $\psi_2\in L^1[t_0,+\infty)$. Consequently, we get that as $t\to +\infty$
$$y(t)=(1+o(1))e^{\lambda(t-t_0)}\exp\left(\int_{t_0}^t\left[ \theta_1(s)+ \theta_2(s)\right]\,ds\right).$$
Note that as $t\to +\infty$
\begin{equation*}
    \begin{split}
        \int_{t_0}^t\left[ \theta_1(s)+ \theta_2(s))\right]\,ds  = &\, c+ o(1)+(-1)^n\prod_{k=1}^{n-1}(\lambda_k-\lambda)^{-1}\int_{t_0}^t\bigg[ P(r;\lambda)(s) +\mathcal{L}(\cdot,\theta_1(s))\\ 
        & \, +\sum_{i=2}^n \tilde{a}_i h_{2,i-1}(\theta_1(s),\theta_1'(s),\dots,\theta_1^{(i-2)}(s)) \bigg]\,ds.
    \end{split}
\end{equation*}
Assuming that $p\in (3,4]$, from latter arguments follows that $z=\theta_1+\theta_2+\psi_2$, with $\theta_i\in L^{p/i}[t_0,+\infty)$ for $i=1,2$ and $\psi_2\in L^{p/3}[t_0,+\infty)$. Let us stress that $\theta_1$ and $\theta_2$ depend only on perturbations $r_i$'s and but $\psi_2$ depends on $z$. Hence, we replace $\psi_1=\theta_2+\psi_2$ in the integral equation for $z$ and decompose the expressions according its order of integrability taking into account \eqref{dfz}. Thus, it follows that $z=\theta_1+\theta_2+\theta_3+\psi_3$, with $\theta_i\in L^{p/i}[t_0,+\infty)$ for $i=1,2,3$ and $\psi_3\in L^{1}[t_0,+\infty)$. The same procedure works out for $p\in(4,5]$, $p\in(5,6]$ and so on. Expressions can be found explicitly in terms of perturbations $r_i$'s, so that $\displaystyle z=\sum_{i=1}^m\theta_i + \psi_m,$
where $\theta_i\in L^{p/i}[t_0,+\infty)$ for $i=1,\dots,m$ are functions depending only on $r_i$, $i=0,\dots,n-1$ and $\psi_m\in L^1[t_0,+\infty)$. Therefore, we deduce the following statement.

\begin{Theorem}\label{hllp}
Assume that $r_i\in L^p[t_0,+\infty)$, for all $i=0,\dots,n-1$ with $p\in(m,m+1]$ for some $m\in\mathbb{N}$, $m\geq 2$. Then there exists a solution $y_\lambda$ for the equation \eqref{poinca} satisfying \eqref{dlp} and as $t\to + \infty$ 
$$y_\lambda(t)=(1+o(1))e^{\lambda(t-t_0)}\exp\left(\int_{t_0}^t\left[ \theta_1(s)+\cdots + \theta_m(s)\right]\,ds\right),$$
where $\theta_i\in L^{p/i}[t_0,+\infty)$ for $i=1,\dots,m$ are functions depending on $r_i$, $i=0,\dots,n-1$, $\theta_1=-G[P(r;\lambda)]$, every $\theta_k$ for $k=2,\dots,n-1$ depends on $\theta_1,\dots,\theta_{k-1}$ and can be found following a recurrence.
\end{Theorem}

In next section, we will present main ideas of the proof, addressing the case $n=5$.

\section{Comments and examples}\label{exams}

In this section, we apply our results to a general equation of order 5 and also a particular choice of $a_i$'s and perturbations $r_i$'s.

\subsection{The case n=5}

Here we show how the complete Bell's polynomial permits to get explicitly the equation for $z$ in case $n=5$. 
In this case we have
$$
y^{(5)}(t)+\sum_{i=0}^{4}(a_i+r_i(t))y^{(i)}(t)=0\,.
$$
In order to reduce the equation for $z$, we will use the formula \eqref{poincaz}. The complete Bell polynomials can be computed recursively using the relation \eqref{complebell}. Thus, from \eqref{bellpols} 
and Theorem \ref{theo1} one has that 
 \begin{equation*}
 \begin{split}
 P(r(t);\lambda)&=r_4(t)\lambda^4+ r_3(t)\lambda^3+r_2(t)\lambda^2+r_1(t)\lambda+r_0(t) \,, \\
\mathcal{ D}z(t) &= z^{(4)}(t)+( 5\lambda+ a_4)z^{(3)}+(10\lambda^2+4a_4\lambda+a_3)z^{(2)}(t)  \\
 &\quad +(10\lambda^3+6\lambda^2a_4+3\lambda a_3+a_2)z'(t)+(5\lambda^4+4a_4\lambda^3+3a_3\lambda^2+2a_2\lambda+a_1)z(t)  \\
\mathcal{L}(t,z) &=\, r_4(t)z^{(3)}(t) +\big(4r_4(t)\lambda+r_3(t)\big)z^{(2)}(t)+ (6\lambda^2r_4(t)+3\lambda r_3(t)+r_2(t))z'(t) \\
&\quad \quad +  (4r_4(t)\lambda^3+3r_3(t)\lambda^2+2r_2(t)\lambda+r_1(t))z(t) 
 \end{split}
 \end{equation*}
Thus, if we subtract $z',z^{(2)},z^{(3)},z^{(4)}$ to the four last terms of the computation above, one obtains the non-linear part $\mathcal{F}(t,Z(t)):=\mathcal{F}(t,z,z',z^{(2)},z^{(3)},z^{(4)})$\,. Explicitly 
\begin{equation*}
\begin{split}
\mathcal{F}(t,Z)=&  z^2(t)\sum_{j=0}^{3}{j+2\choose j}(a_{j+2}+r_{j+2}(t))\lambda^j  \\
&\quad +\big(z(t)^3+3z(t)z'(t)\big)\sum_{j=0}^{2}{j+3\choose j}(a_{j+3}+r_{j+3}(t))\lambda^j   \\
&\quad +\big(z(t)^4+6z(t)^2z'(t)+4z(t)z^{(2)}(t)+3z'(t)^2\big)\sum_{j=0}^{1}{j+2\choose j}(a_{j+4}+r_{j+4}(t))\lambda^j \\
&\quad +  z(t)^5+10z(t)^3z'(t)+15z'(t)^2z(t) + 10[z(t)]^2z^{(2)}(t)+10z^{(2)}(t)z'(t)+5z^{(3)}(t)z(t)\,.
\end{split}
\end{equation*}
We simplify notation for $i=2,3,4,5$
$$\tilde{a}_i=\sum_{j=0}^{5-i}{i+j \choose j} a_{i+j}\lambda^j\qquad\text{and}\qquad \tilde{r}_i(t)=\sum_{j=0}^{4-i}{i+j \choose j} \lambda^j r_{i+j}(t).$$
Note that $\tilde a_5=1$ and $\tilde r_4=r_4$. Thus, we have the following integral equation for $z$
\begin{equation}\label{eizn5}
\begin{split}
z(t)= - G\Big[ & P(r;\lambda)+ \mathcal{L}(\cdot,z) +  \tilde a_2 z^2 + 3\tilde a_3 zz' + \tilde a_4 \big(4zz'' + 3[z']^2\big) + 5zz^{(3)} + 10 z'z'' \\
& + \tilde a_3 z^3 + 6\tilde a_4 z^2z' + 10 z^2z'' + 15 z[z']^2 + \tilde r_2 z^2 + 3\tilde r_3 zz' + r_4\big( 4zz'' + 3[z']^2\big)\\
& + \tilde a_4 z^4 + 10 z^3z' + \tilde r_3 z^3 + 6 r_4 z^2z''  + r_4 z^4+ z^5\Big](t) .
\end{split}
\end{equation}

1. Note that $zz'$, $z'z''$, $z^2z'$ and $z^3z'$ are conditionally integrable. Hence, under assumptions \eqref{gpr} and \eqref{cl0} we find the asymptotic formula as $t\to+\infty$
\begin{equation*}
    \begin{split}
        y(t)=(1+o(1))& e^{\lambda(t-t_0)}\\
        \times& \exp\bigg(-\prod_{j=1}^4(\lambda_j-\lambda)^{-1}\int_{t_0}^t [P(r;\lambda)(s)+\mathcal{L}(s,z(s))+\widehat{\mathcal{F}}(s,Z(z))]\,ds\bigg)\,,
    \end{split}
\end{equation*}
where
\begin{equation*}
    \begin{split}
        \widehat{\mathcal{F}}(t,Z)=&\, \tilde a_2 z^2 + \tilde a_4 \big(4z z'' + 3[z']^2\big) + 5zz^{(3)} + \tilde a_3 z^3 + 10 z^2 z'' + 15 z[z']^2 + \tilde r_2(t) z^2 + 3\tilde r_3(t) z z'\\
        & + r_4(t)\big( 4zz'' + 3[z']^2\big) + \tilde a_4 z^4 + \tilde r_3(t) z^3 + 6 r_4(t) z^2 z'' + r_4(t) z^4+ z^5 
    \end{split}
\end{equation*}

\medskip

2. Now, we develop Theorem \ref{hllp} for $n=5$. Assume that $r_i\in L^p[t_0,+\infty)$ for all $i=0,\dots,4$ with $p\in(m,m+1]$ for some $m\in\mathbb{N}$, $m\ge 2$. We know that $z\in L^p[t_0,+\infty)$ and $z=\theta_1 + \psi_1$, where $\theta_1=-G[P(r;\lambda)]$ with $\theta_1^{(i)}\in L^p$ and $\psi_1^{(i)}\in L^{p/2}$ for all $i=0,\dots,4$. In general, suppose that
$$z=\varphi_k+\psi_{k+1},\quad \text{with}\quad \varphi_k=\sum_{l=1}^k\theta_l,\ \ k<m-1,$$
where $\theta_l$ are given by
\begin{equation*}
\theta_2= - G\Big[ \mathcal{L}(\cdot,\theta_1) +  \tilde a_2 \theta_1^2 + 3\tilde a_3 \theta_1\theta_1' + \tilde a_4 \big(4\theta_1\theta_1'' + 3[\theta_1']^2\big) + 5\theta_1\theta_1^{(3)} + 10 \theta_1'\theta_1''\Big]
\end{equation*}
\begin{equation*}
\begin{split}
\theta_3= - G\Big[& \mathcal{L}(\cdot,\theta_2) +  2\tilde a_2 \theta_1\theta_2 + 3\tilde a_3 (\theta_1\theta_2' +\theta_1'\theta_2) + \tilde a_4 \big(4\{\theta_1\theta_2'' + \theta_1''\theta_2\} + 6 \theta_1'\theta_2'\big) \\
&+ 5\{\theta_1\theta_2^{(3)} + \theta_1^{(3)}\theta_2\} + 10( \theta_1'\theta_2'' + \theta_1''\theta_2') + \tilde a_3 \theta_1^3 + 6\tilde a_4 \theta_1^2\theta_1' + 10 \theta_1^2\theta_1'' + 15 \theta_1[\theta_1']^2\\
& + \tilde r_2 \theta_1^2 + 3\tilde r_3 \theta_1\theta_1' + r_4(4\theta_1\theta_1'' + 3[\theta_1']^2)\Big]
\end{split}
\end{equation*}
\begin{equation*}
\begin{split}
\theta_4= - G\Big[& \mathcal{L}(\cdot,\theta_3) + \sum_{i=1}^3 \big(\tilde a_2 \theta_i\theta_{4-i} + 3\tilde a_3 \theta_i\theta_{4-i}' + \tilde a_4 \big\{4 \theta_i\theta_{4-i}'' + 3 \theta_i'\theta_{4-i}'\big\} + 5 \theta_i\theta_{4-i}^{(3)} + 10\theta_i'\theta_{4-i}''  \big)\\
& + \sum_{i+j+n=4}\big(\tilde a_3 \theta_i\theta_j\theta_n + 6\tilde a_4 \theta_i\theta_j\theta_n' + 10 \theta_i\theta_j\theta_n'' + 15 \theta_i\theta_j'\theta_n' \big)+ 2\tilde r_2 \theta_1\theta_2 + 3\tilde r_3\{ \theta_1\theta_2'\\
&+\theta_1'\theta_2\} + \tilde r_4( 4\{\theta_1\theta_2'' + \theta_1'' \theta_2\} + 6\theta_1'\theta_2'  ) + \tilde a_4\theta_1^4 + 10 \theta_1^3\theta_1' + \tilde r_3\theta_1^3 + 6r_4\theta_1^2\theta_1''\Big]
\end{split}
\end{equation*}
and in general, for $l\ge 5$
\begin{equation*}
\begin{split}
\theta_l = - G\Big[& \mathcal{L}(\cdot,\theta_{l-1}) + \sum_{i=1}^{l-1} \big(\tilde a_2 \theta_i\theta_{l-i} + 3\tilde a_3 \theta_i\theta_{l-i}' + \tilde a_4 \big\{4 \theta_i\theta_{l-i}'' + 3 \theta_i'\theta_{l-i}'\big\} \\
&+ 5 \theta_i\theta_{l-i}^{(3)} + 10\theta_i'\theta_{l-i}''  \big) + \sum_{i+j+n=l}\big(\tilde a_3 \theta_i\theta_j\theta_n + 6\tilde a_4 \theta_i\theta_j\theta_n' + 10 \theta_i\theta_j\theta_n'' + 15 \theta_i\theta_j'\theta_n' \big)\\
&+ \sum_{i=1}^{l-2}\big( \tilde r_2 \theta_i\theta_{l-i-1} + 3\tilde r_3  \theta_i\theta_{l-i-1}' + \tilde r_4\{ 4 \theta_i\theta_{l-i-1}'' + 3 \theta_i'\theta_{l-i-1}' \}\big) \\
& + \sum_{i+j+m+n=l}\big( \tilde a_4\theta_i\theta_j\theta_m\theta_n + 10 \theta_i\theta_j\theta_m\theta_n'\big) + \sum_{i+j+n=l-1} \big(\tilde r_3\theta_i\theta_j\theta_n + 6r_4\theta_i\theta_j\theta_n''\big)\\
& + \sum_{i+j+m+n=l-1} \tilde r_4\theta_i\theta_j\theta_m\theta_n +\sum_{i+j+m+n+q=l}\theta_i\theta_j\theta_m\theta_n\theta_q\Big]
\end{split}
\end{equation*}
$\theta_l^{(i)}\in L^{p/l}$ and $\psi_{k+1}^{(i)}\in L^{p/(k+1)}$ for all $i=0,\dots,4$. Replacing it in \eqref{eizn5}, we get that $$\mathcal{L}(\cdot,\varphi_k+\psi_{k+1})=\mathcal{L}(\cdot,\varphi_k ) +\mathcal{L}(\cdot, \psi_{k+1}) = \sum_{l=1}^k\mathcal{L}(\cdot,\theta_l ) +\mathcal{L}(\cdot, \psi_{k+1}),$$ 
$\mathcal{L}(\cdot, \theta_{l}) \in L^{p/(l+1)}$, $\mathcal{L}(\cdot, \psi_{k+1}) \in L^{p/(k+2)}$ and $\mathcal{F}(\cdot,\varphi_k+\psi_{k+1}) - \mathcal{F}(\cdot,\varphi_k)\in L^{p/(k+2)}$. Hence, we compute $\mathcal{F}(\cdot,\varphi_k)$ by using the definition of $\varphi_k$. Note that
$$\varphi_k^2=\sum_{l=1}^k\underbrace{\sum_{i=1}^{l} \theta_i\theta_{l-i+1}}_{\in L^{p/(l+1)}} + \underbrace{\sum_{i+j>k+1}\theta_i\theta_j}_{\in L^{p/(k+2)}},\qquad \varphi_k\varphi_k'=\sum_{l=1}^k\underbrace{\sum_{i=1}^{l} \theta_i\theta_{l-i+1}'}_{\in L^{p/(l+1)}} + \underbrace{\sum_{i+j>k+1}\theta_i\theta_j'}_{\in L^{p/(k+2)}},$$
and similar expressions for remain terms of degree 2: $\varphi_k\varphi_k''$, $[\varphi_k']^2$, $\varphi_k\varphi_k^{(3)}$ and $\varphi_k'\varphi_k''$ can be found. For terms of degree 3, we obtain that
\begin{equation*}
    \begin{split}
        \varphi_k^3=&\sum_{l=2}^{k} \underbrace{\sum_{i+j+n=l+1}\theta_i\theta_{j}\theta_n}_{\in L^{p/(l+1)}} + \underbrace{\sum_{i+j+n>k+1}\theta_i\theta_j\theta_n}_{\in L^{p/(k+2)}},\\
        \varphi_k^2\varphi_k'=&\sum_{l=2}^{k}\underbrace{ \sum_{i+j+n=l+1}\theta_i\theta_{j}\theta_n'}_{\in L^{p/(l+1)}} + \underbrace{ \sum_{i+j+n>k+1}\theta_i\theta_j\theta_n'}_{\in L^{p/(k+2)}},
    \end{split}
\end{equation*}
and similar expressions for remain terms of degree 3: $\varphi_k^2\varphi_k'$ and $\varphi_k[\varphi_k']^2$ can be found. Also, we have that
$$\tilde r_2 \varphi_k^2=  \sum_{l=2}^k\underbrace{\sum_{i=1}^{l-1} \tilde r_2\theta_i\theta_{l-i} }_{\in L^{p/(l+1)}}+ \underbrace{ \tilde r_2 \sum_{i+j>k+1}\theta_i\theta_j}_{\in L^{p/(k+2)}},$$
and similar for $\tilde r_3\varphi_k\varphi_k'$ and $r_4\{4\varphi_k\varphi_k'' + 3[\varphi_k']^2\}$ can be found. Next term of degree 4 is
$$\varphi_k^4=\sum_{l=3}^{k}\underbrace{ \sum_{i+j+m+n=l+1}\theta_i\theta_{j}\theta_m\theta_n}_{\in L^{p/(l+1)}} + \underbrace{ \sum_{i+j+m+n>k+1}\theta_i\theta_j\theta_m\theta_n}_{\in L^{p/(k+2)}} $$
and similar for $\varphi_k^3\varphi_k'$; also,
$$\tilde r_3 \varphi_k^3=\sum_{l=3}^{k}\underbrace{ \sum_{i+j+n=l}\tilde r_3\theta_i\theta_{j}\theta_n}_{\in L^{p/(l+1)}} + \underbrace{\tilde r_3 \sum_{i+j+n>k+1}\theta_i\theta_j\theta_n}_{\in L^{p/(k+2)}}$$
and similar expression can be found for $r_4 \varphi_k^2\varphi_k''$. The last two terms of degree 5 are
\begin{equation*}
    \begin{split}
        \varphi_k^5=&\sum_{l=4}^{k}\underbrace{ \sum_{i+j+m+n+q=l+1}\theta_i\theta_{j}\theta_m\theta_n\theta_q}_{\in L^{p/(l+1)}} + \underbrace{\sum_{i+j+m+n+q>k+1}\theta_i\theta_j\theta_m\theta_n\theta_q}_{\in L^{p/(k+2)}} \\
        r_4 \varphi_k^4 = & \sum_{l=4}^{k}\underbrace{ \sum_{i+j+m+n=l}r_4\theta_i\theta_{j}\theta_m\theta_n}_{\in L^{p/(l+1)}} + \underbrace{r_4 \sum_{i+j+m+n>k+1}\theta_i\theta_j\theta_m\theta_n}_{\in L^{p/(k+2)}}
    \end{split}
\end{equation*}
Then, after tedious computations we obtain that
\begin{equation*}
    \begin{split}
        \mathcal{F}(\cdot,\varphi_k)=&\,(\text{terms in $L^{p/2}$, $l=1$}:)\, \Big\{ \tilde a_2 \theta_1^2 + 3\tilde a_3 \theta_1\theta_1' + \tilde a_4 \big(4\theta_1\theta_1'' + 3[\theta_1']^2\big) + 5\theta_1\theta_1^{(3)} + 10 \theta_1'\theta_1''\Big\}\\[0.4cm]
        &+(\text{terms in $L^{p/3}$, $l=2$}:)\,\Big\{ 2\tilde a_2 \theta_1\theta_2 + 3\tilde a_3 (\theta_1\theta_2' +\theta_1'\theta_2) + \tilde a_4 \big(4\{\theta_1\theta_2'' + \theta_1''\theta_2\} + 6 \theta_1'\theta_2'\big) \\
&+ 5\{\theta_1\theta_2^{(3)} + \theta_1^{(3)}\theta_2\} + 10( \theta_1'\theta_2'' + \theta_1''\theta_2') + \tilde a_3 \theta_1^3 + 6\tilde a_4 \theta_1^2\theta_1' + 10 \theta_1^2\theta_1'' + 15 \theta_1[\theta_1']^2\\
& + \tilde r_2 \theta_1^2 + 3\tilde r_3 \theta_1\theta_1' + r_4(4\theta_1\theta_1'' + 3[\theta_1']^2)\Big\}\\[0.3cm]
&+ (\text{terms in $L^{p/4}$, $l=3$}:)\,\Big\{ \sum_{i=1}^3 \big(\tilde a_2 \theta_i\theta_{4-i} + 3\tilde a_3 \theta_i\theta_{4-i}' + \tilde a_4 \big\{4 \theta_i\theta_{4-i}'' + 3 \theta_i'\theta_{4-i}'\big\} \\
&+ 5 \theta_i\theta_{4-i}^{(3)} + 10\theta_i'\theta_{4-i}''  \big) + \sum_{i+j+n=4}\big(\tilde a_3 \theta_i\theta_j\theta_n + 6\tilde a_4 \theta_i\theta_j\theta_n' + 10 \theta_i\theta_j\theta_n'' + 15 \theta_i\theta_j'\theta_n' \big)\\
&+ 2\tilde r_2 \theta_1\theta_2 + 3\tilde r_3\{ \theta_1\theta_2'+\theta_1'\theta_2\} + \tilde r_4( 4\{\theta_1\theta_2'' + \theta_1'' \theta_2\} + 6\theta_1'\theta_2'  ) \\
& + \tilde a_4\theta_1^4 + 10 \theta_1^3\theta_1' + \tilde r_3\theta_1^3 + 6r_4\theta_1^2\theta_1'' \Big\}\\[0.3cm]
        &+(\text{terms in $L^{p/l}$, $l=5,\dots,k+1$}:)\,\Big\{\sum_{l=5}^{k+1}\bigg[ \sum_{i=1}^{l-1} \big(\tilde a_2 \theta_i\theta_{l-i} + 3\tilde a_3 \theta_i\theta_{l-i}' \\
        &+ \tilde a_4 \big\{4 \theta_i\theta_{l-i}'' + 3 \theta_i'\theta_{l-i}'\big\} + 5 \theta_i\theta_{l-i}^{(3)} + 10\theta_i'\theta_{l-i}''  \big) \\
        &+ \sum_{i+j+n=l}\big(\tilde a_3 \theta_i\theta_j\theta_n + 6\tilde a_4 \theta_i\theta_j\theta_n' + 10 \theta_i\theta_j\theta_n'' + 15 \theta_i\theta_j'\theta_n' \big)\\
&+ \sum_{i=1}^{l-2}\big( \tilde r_2 \theta_i\theta_{l-i-1} + 3\tilde r_3  \theta_i\theta_{l-i-1}' + \tilde r_4\{ 4 \theta_i\theta_{l-i-1}'' + 3 \theta_i'\theta_{l-i-1}' \}\big) \\
& + \sum_{i+j+m+n=l}\big( \tilde a_4\theta_i\theta_j\theta_m\theta_n + 10 \theta_i\theta_j\theta_m\theta_n'\big) + \sum_{i+j+n=l-1} \big(\tilde r_3\theta_i\theta_j\theta_n + 6r_4\theta_i\theta_j\theta_n''\big)\\
&+\sum_{i+j+m+n+q=l}\theta_i\theta_j\theta_m\theta_n\theta_q + \sum_{i+j+m+n=l-1} \tilde r_4\theta_i\theta_j\theta_m\theta_n \bigg] \Big\} +\text{terms in $L^{p/(k+2)}$}.
    \end{split}
\end{equation*}

\noindent Therefore, we conclude that
\begin{equation*}
    \begin{split}
        -&G\Big[\mathcal{L}(\cdot,\varphi_k+\psi_{k+1}) + \mathcal{F}(\cdot,\varphi_k+\psi_{k+1})\Big]\\
        =&\, -G\Big[\mathcal{L}(\cdot,\varphi_k) + \mathcal{F}(\cdot,\varphi_k)\Big] -G\Big[\mathcal{L}(\cdot,\psi_{k+1}) + \mathcal{F}(\cdot,\varphi_k+\psi_{k+1}) -  \mathcal{F}(\cdot,\varphi_k)\Big]\\
        =&\, \sum_{l=2}^{k+1}\theta_l + \text{terms in $L^{p/(k+2)}$},
    \end{split}
\end{equation*}
so that, $\displaystyle z=\sum_{i=1}^{m}\theta_i + \psi_{m+1}$. Thus, we have deduced asymptotic formula in Theorem \ref{hllp}.

\medskip
3. In particular, if $m=2$, that is to say $p\in(2,3]$, then $z=\theta_1+\theta_2+\psi_3$ with \linebreak $\theta_1=-G[P(r;\lambda)]$ and
\begin{equation*}
\theta_2= - G\Big[ \mathcal{L}(\cdot,\theta_1) +  \tilde a_2 \theta_1^2 + 3\tilde a_3 \theta_1\theta_1' + \tilde a_4 \big(4\theta_1\theta_1'' + 3[\theta_1']^2\big) + 5\theta_1\theta_1^{(3)} + 10 \theta_1'\theta_1''\Big],
\end{equation*}
$\theta_l^{(i)}\in L^{p/l}$, $l=1,2$ and $\psi_3\in L^1$. Integrating $\theta_1$ and $\theta_2$ and taking into account that $\theta_1\theta_1'$ and $\theta_1'\theta_1''$ are conditionally integrable, we find the following asymptotic formula in terms of $\theta_1$
\begin{equation*}
    \begin{split}
        y(t)=(1&+o(1)) e^{\lambda(t-t_0)}\exp\bigg(-\prod_{j=1}^4(\lambda_j-\lambda)^{-1}\int_{t_0}^t \big[P(r;\lambda)(s) + \mathcal{L}(s,\theta_1(s)) + \tilde a_2 (\theta_1(s))^2\big]\bigg)\\   
        \times&\exp\bigg(-\prod_{j=1}^4(\lambda_j-\lambda)^{-1}\int_{t_0}^t \big[\tilde a_4(4\theta_1(s)\theta_1''(s)+3[\theta_1'(s)]^2) + 5\tilde a_5 \theta_1(s)\theta_1^{(3)}(s) \big]\,ds\bigg)\,,
    \end{split}
\end{equation*}

\subsection{An explicit equation of degree 5}

\hfill

\medskip
\noindent Consider the equation 
$$
y^{(5)}+(r_3(t)-5)y^{(3)}+(r_1(t)+4)y^{(1)}+r_0(t)y=0
$$

Note that the polynomial $P(a,x)$ associated to this equations corresponds to
$$
P(a;x) = x^5-5x^3+4x = x(x-1)(x+1)(x+2)(x-2)
$$
Let us denote by $\lambda$ a root of $P(a,x)$. Note that in this case we have $a_4=a_2=a_0=0$\,.

We make the change of variable $z=y^{'}/y-\lambda$ and using the previous subsection we obtain the following equation
%


\begin{equation}
\begin{split}
z(t) =  & -G\Big[  P(r;\lambda)+ \mathcal{L}(\cdot,z) +  (10\lambda^3-15\lambda) z^2 + (30\lambda^2-15) zz' + 5\lambda \big(4zz'' + 3[z']^2\big)\\ 
&  + 5zz^{(3)}+ 10 z'z'' + (30\lambda^2-15)z^3 + 30\lambda z^2z' + 10 z^2z'' + 15 z[z']^2 + 3\lambda r_3 z^2 \\
&  + 3r_3 zz' + r_4\big( 4zz'' + 3[z']^2\big)+ 5\lambda z^4 + 10 z^3z' + r_3 z^3 + 6 r_4 z^2z'' + z^5 + r_4 z^4\Big](t) .
\end{split}
\end{equation}

where $P(r;\lambda)= r_3(t)\lambda^3+r_1(t)\lambda+r_0(t)$
and  $\mathcal{L}(\cdot,z) = r_3(t)z^{(2)}(t)+ (3\lambda r_3(t)+r_3(t))z'(t)+  (3r_3(t)\lambda^2+r_1(t))z(t)$

In order to apply Theorem \ref{orederzun}, we estimate the left hand side of inequality \eqref{cl0}. We can compute
$$
\begin{aligned}
\frac{1}{1!}\frac{\partial}{\partial x}P(r(t);\lambda) &= 3r_3(t)\lambda^2+r_1(t)\,, &
\frac{1}{2!}\frac{\partial^2}{\partial x^2}P(r(t);\lambda) &= 3r_3(t)\lambda\,, \\
\frac{1}{3!}\frac{\partial^3}{\partial x^3}P(r(t);\lambda) &= r_3(t)\,, &
\frac{1}{4!}\frac{\partial^4}{\partial x^4}P(r(t);\lambda) &= 0\,. \\
\end{aligned}
$$
Using the notation of section 2 $\gamma_j=\lambda_j-\lambda$\,. For instance, if we consider the case $\lambda=0$ we obtain that $\gamma_1=1$, $\gamma_2=-1$, $\gamma_3=2$ and $\gamma_4=-2$\,, then $\tilde\alpha_1=\tilde\alpha_2=4$ and $\tilde\alpha_3=\tilde\alpha_4=15$, and $|\Gamma_1|=|\Gamma_2|=6$ and $|\Gamma_3|=|\Gamma_4|=6$\,.

In general, if $\lambda\in\{0,1,-1,2,-2 \}$ we can see that  $|\alpha_j|=|\Re{\gamma_j}|\leq 4$ for all $j$\,. Hence $\tilde\alpha_j\leq 85$ for $j=0,1,2,3,4$\,. Also, $|\Gamma_j|\geq1$, hence we obtain that
$$
\frac{\tilde\alpha_j}{|\Gamma_j|}\leq 85 \qquad j=0,1,2,3,4\,.
$$

Therefore, we obtain
\begin{equation*}
\begin{split}
 	\sum_{j=1}^{n-1}\sum_{k=1}^{n-1} \frac{\tilde{\alpha}_j}{k! \left|\Gamma_j\right|}  \bigg\| I_{\alpha_j}\Big[ \frac{\partial^k}{\partial x^k}P(r(\cdot);\lambda) \Big] \bigg\|_\infty 
	& \leq
	85\sum_{j=0}^4 |3\lambda^2+3\lambda+1|\big\| I_{\alpha_j}\big(r_3\big) \big\|_\infty+\big\|I_{\alpha_j}\big(r_1\big)\big\|_\infty \\
	&\leq
	85\sum_{j=0}^4 19\big\| I_{\alpha_j}\big(r_3\big) \big\|_\infty+\big\|I_{\alpha_j}\big(r_1\big)\big\|_\infty\,,
\end{split}
\end{equation*}

since, for any root $\lambda$ of $P(a;\lambda)$, we have that $|3\lambda^2+3\lambda+1| \leq 19$\,. Thus, the theorem \ref{orederzun} holds if
\begin{equation}\label{examexte}
  \begin{aligned}
  \sum_{i=0}^{n-2}\big|G[P(r;\lambda)]^{(i)}(t)\big|\to 0\quad\text{as }\, t\to +\infty\,, \text{and } \\
   85\sum_{j=0}^4 19\big\| I_{\alpha_j}\big(r_3\big) \big\|_\infty+\big\|I_{\alpha_j}\big(r_1\big)\big\|_\infty <1\,.
  \end{aligned}
\end{equation}

Also note that for any election of $\lambda$ one has that $1\leq|\alpha_j|$\,, hence the second part of theorem \ref{orederzun} holds if there exists a $0<\beta<1$ such that 

\begin{equation}\label{examdos}
85\sum_{j=0}^4 19\big\| I_{\alpha_j-\sgn(\alpha_j)\beta}\big(r_3\big) \big\|_\infty+\big\|I_{\alpha_j-\sgn(\alpha_j)\beta}\big(r_1\big)\big\|_\infty < \frac{1}{2}
\end{equation}
Note that, if  $r_i\to 0$ as $t\to \infty$, for $i=1,2$ one can guarantee the existence of $t_0$ such that inequalities \eqref{examexte} and \eqref{examdos} hold. 
Moreover, under this hypothesis also theorem \ref{base} holds, i.e., there exists a fundamental system of solutions $y_i$, $i=1,2,3,4,5$ of equation  \eqref{poinca} such that
 $$
\frac{y_i^{(j)}(t)}{y_i(t)}=\lambda_i^{j} + O\big((I_{\beta}+I_{-\beta})[\lambda_i^3r_3+\lambda_i r_1+r_0](t)\big)
	\quad \text{for } j=0,1,\dots,n-2\,.
$$

On the other hand, if also it holds that  $\lambda_i^3r_3+\lambda_i r_1+r_0\in L^1[t_0,\infty)$ (in particular, if $r_i\in L^1[t_0,\infty)$\,, for $i=0,1,3$) one can apply Theorem \ref{pplev} to obtain
$$
y_i(t)=\bigg[1+O\left(\int_{t}^{+\infty}(I_{\beta}+I_{-\beta})[\lambda_i^3r_3+\lambda_i r_1+r_0](s)\, ds \right)\bigg] e^{\lambda_i(t-t_0)}.
$$
Note that $|I_{\pm\beta}[P(r,\lambda)](t)|\leq P(r(t),\lambda)$\,, for $i=1,2$. Thus, $I_{\pm\beta}[r_i]\in L^1$ and one can  rewrite the previous equation just by
$$
y_i(t)=(1+O(1))e^{\lambda_i(t-t_0)}\,.
$$

\medskip 
An interesting case is the following case $r_{3}(t)=\frac{1}{\sqrt[3]{t^2}}, r_{1}(t)=\frac{1}{\sqrt[3]{t^2+1}}$ and $r_{0}=r_{1} $\,.

We  consider $r_{i} \in L^{2}[t_0,\infty)$,
$i=0,1,2$\,, thus  $P(r,\lambda)\in L^{2}[t_0,\infty)$\,. Moreover, the hypothesis of Theorem \ref{teots} are satisfied. In order to obtain an asymptotic formula like the one in Theorem \ref{teots}, first we check the following properties of our equation. In this case one has that 
$$
P(r(t);\lambda) \to 0 \;\mbox{ then }\; \theta(t)= G_{\gamma_j}[P(r;\lambda)](t)\to 0 \,,\;\mbox{ as }t\to\infty\,.
$$
Moreover, since $|\theta(t)|=|G_{\gamma_j}[P(r;\lambda)](t)|\leq |P(|r(t)|,|\lambda|)|$, one has that $\theta\in L^2[t_0,\infty)$. Indeed, at least one has that $\theta$ is $O(t^{-2/3})$. Now, since 
$$
\begin{aligned}
R(\cdot,\theta)&= \mathcal{L}(\cdot,\theta) +  (10\lambda^3-15\lambda) \theta^2 + (30\lambda^2-15) \theta\theta' + 5\lambda \big(4\theta\theta'' + 3[\theta']^2\big)\\ 
&  + 5\theta\theta^{(3)}+ 10 \theta'\theta'' + (30\lambda^2-15)\theta^3 + 30\lambda \theta^2\theta' + 10 \theta^2\theta'' + 15 \theta[\theta']^2 + 3\lambda r_3 \theta^2 \\
&  + 3r_3 \theta\theta' + r_4\big( 4\theta\theta'' + 3[\theta']^2\big)+ 5\lambda \theta^4 + 10 \theta^3\theta' + r_3 \theta^3 + 6 r_4 \theta^2\theta'' + \theta^5 + r_4 \theta^4 \,,
\end{aligned}
$$
thus $R(\cdot,\theta)=O(t^{-4/3})$\,. With this information, the result of Theorem \ref{teots} implies that
\begin{equation*}
        y(t)= (1+\varepsilon(t))e^{\lambda(t-t_0)} \exp\left(-\prod_{k=1}^{4}(\lambda_k-\lambda)^{-1}\int_{t_0}^tP(r;\lambda)(s)\,ds\right),
\end{equation*}
with 
$$
\varepsilon(t)=O\bigg(\sum_{j=1}^{n-1} I_{\alpha_j}[\mathcal{L}(\cdot,u)+ \tilde{\mathcal{F}}(\cdot ,U)] +\int_t^{+\infty}\left[\big|\mathcal{L}(s,u(s))\big|+ \big|\tilde{\mathcal{F}}(s,U(s))\big|\right]ds\bigg).
$$

Also, the the function $z-\theta$ satisfies the equation \eqref{eu} and $u^{(i)}=O\big((I_\beta+I_{\beta})[R(\cdot,\theta )]\big))$   for $i=0,1,\dots,(n-2)$\,. It follows that, at least, $u^{(i)}=O(t^{-4/3})$\,. Moreover, we can conclude also that 
$
\mathcal{L}(t,u(t))=O(t^{-2}),
$
and $\tilde{\mathcal{F}}(\cdot,U)=O(t^{-2})$
$$
\varepsilon(t)=O\big(1/t)\,.
$$
therefore, we have that 
$$
\begin{aligned}
  y(t)= (1+O(1/t))ce^{\lambda(t-t_0)} &\exp\left(\prod_{k=1}^{4}(\lambda_k-\lambda)^{-1}\left(\frac{\lambda^3+1}{3}\right)\sqrt[3]{t} \right)
  \\
  \times& \exp\left(-\prod_{k=1}^{4}(\lambda_k-\lambda)^{-1}\frac{\lambda}{8}\left(t \sqrt{t^{2}+1}\left(2 t^{2}+5\right)+3 \ln(t+\sqrt{t^2+1})\right) \right)
  \end{aligned}
$$
where $c$ is a constant.
For example, considering $\lambda=1$ we obtain that 
$$
  y(t)= (1+O(1/t))ce^{\lambda(t-t_0)} \exp\left(\frac{\sqrt[3]{t}}{9} +\frac{t \sqrt{t^{2}+1}\left(2 t^{2}+5\right)}{48} \right)\left(t+\sqrt{t^2+1} \right)^{1/16}\,.
$$

\bigskip
\begin{center}
{\bf Acknowledgements}    
\end{center}
The second author has been supported by grant Fondecyt Regular N$^\circ$ 1201884. The third author has been supported by grant Fondecyt Regular N$^\circ$ 1170466.

\end{document}